\documentclass[1p]{elsarticle}
\usepackage{hyperref}
\usepackage{amssymb}
\usepackage{latexsym}
\usepackage{amscd}
\usepackage{amsthm}
\usepackage{amsfonts}
\usepackage{CJK}
\usepackage{CJKulem}
\usepackage{color}
\usepackage{colortbl}
\usepackage{fancyhdr}
\usepackage{indentfirst}
\usepackage{lastpage}
\usepackage{latexsym}
\usepackage{listings}
\usepackage{multirow}
\usepackage{mathrsfs}
\usepackage{placeins}
\usepackage{titlesec}
\usepackage{ulem}
\usepackage{graphicx} 
\usepackage{amsmath} 
\usepackage{amssymb} 
\usepackage{yfonts}
\usepackage{xypic}
\usepackage{bm}
\usepackage{array}
\usepackage{booktabs}
\usepackage{longtable}
\usepackage{threeparttablex}

\newdimen\AAdi%
\newbox\AAbo%
\def\AAk#1#2{\s_etbox\AAbo=\hbox{#2}\AAdi=\wd\AAbo\kern#1\AAdi{}}%
\def\AAr#1#2#3{\s_etbox\AAbo=\hbox{#2}\AAdi=\ht\AAbo\raise#1\AAdi\hbox{#3}}%
\font\tenmsb=msbm10 at 12pt \font\sevenmsb=msbm7 at 8pt
\font\fivemsb=msbm5 at 6pt
\newfam\msbfam
\textfont\msbfam=\tenmsb \scriptfont\msbfam=\sevenmsb
\scriptscriptfont\msbfam=\fivemsb
\def\Bbb#1{\mathbb{#1}}
\newtheorem{thm}{Theorem}[section]
\newtheorem{main-thm}{Main Theorem}
\newtheorem{lem}[thm]{Lemma}
\newtheorem{cor}[thm]{Corollary}
\newtheorem{pro}[thm]{Proposition}

\theoremstyle{remark}\newtheorem{rem}{Remark}[section]

\renewcommand{\theequation}{\thesection.\arabic{equation}}

\newcommand{\Section}[2]{\setcounter{equation}{0}
	\allowdisplaybreaks
	\section[#1]{#2}}

\def\n{\nabla}

\def\f#1#2{\frac{#1}{#2}}

\def\mc#1{\mathcal{#1}}

\def\td{\tilde}

\def\a{\alpha}
\def\be{\beta}

\def\de{\delta}
\def\De{\Delta}

\def\ep{\varepsilon}

\def\g{\gamma}
\def\k{\kappa}
\def\la{\lambda}

\def\om{\omega}

\def\si{\sigma}

\def\R{\Bbb{R}}
\def\C{\Bbb{C}}
\def\tr{\mbox{tr}}

\def\lan{\langle}
\def\ran{\rangle}
\def\ra{\rightarrow}

\def\ol{\overline}

\def\Id{\mathbf{Id}}

\def\BBt{B\circ B^t}
\def\mf#1{\mathfrak{#1}}
\def\tr{\operatorname{tr}}

\begin{document}
	\title
	{An optimal pinching theorem on compact minimal submanifolds in the Euclidean spheres via eigenvalues of fundamental matrices II}
	
	\author[f]{Huimin Liu}
	\ead{hmliu20@fudan.edu.cn}
	\author[f,m]{Ling Yang}
	\ead{yanglingfd@fudan.edu.cn}
	
	\address[f]{School of Mathematical Sciences, Fudan University, Shanghai, 200433, China}
	\address[m]{Shanghai Center for Mathematical Sciences, Shanghai, 200438, China}


\begin{abstract}
Let $M^n\subset S^{n+m}$ be a compact minimal submanifold of the unit sphere, and let $\la_1\geq\cdots\geq\la_m\geq0$ be the eigenvalues of its fundamental matrix. We prove that $\sum\limits_{\a=1}^{\min\{n,m\}}\la_\a+\la_2\leq n$ forces $M$ to be a totally geodesic subsphere, a generalized Clifford torus, or a Veronese manifold. We first show that equality everywhere in the pinching condition forces the second fundamental form to be parallel. Next, by computing the eigenvalues of the fundamental matrices of irreducible symmetric $R$-spaces, we obtain a complete classification of the equality cases among minimal submanifolds with parallel second fundamental form.
\end{abstract}




	
	\maketitle
	
	\tableofcontents
	
	\renewcommand{\proofname}{\it Proof.}
	
	
	\Section{Introduction}{Introduction}

This paper is a continuation of our previous work \cite{L-Y}, where we studied the rigidity of compact minimal submanifolds in the unit sphere in terms of the eigenvalues of their fundamental matrices.

The pinching theorem in \cite{L-Y} is rooted in the classical rigidity theory of minimal submanifolds in spheres. The first pinching theorem, in terms of the squared norm $|B|^2$, is due to J. Simons \cite{S}, and its equality case was classified by Chern--do Carmo--Kobayashi \cite{C-D-K} and B. Lawson \cite{L}. For $m\geq2$, Li--Li \cite{L-L} and Chen--Xu \cite{C-X} proved rigidity under the condition $|B|^2\leq\f{2n}{3}$. 

Z. Q. Lu \cite{Lu} then introduced the eigenvalues of the fundamental matrix into this problem.
Let $M^n\subset S^{n+m}$ be a compact connected minimal submanifold with $n\geq2$, and let $B$ denote its second fundamental form. For a local orthonormal normal frame $\{\nu_1,\cdots,\nu_m\}$, let $A^\a$ be the shape operator associated with $\nu_\a$. The symmetric matrix
$$
S=(S_{\a\be}),\qquad S_{\a\be}=\tr(A^\a A^\be),
$$
is called the \emph{fundamental matrix} of $M$. Its eigenvalues
$$
\la_1\geq\cdots\geq\la_m\geq0
$$
are independent of the choice of orthonormal frames and satisfy
$$
|B|^2=\tr S=\sum_{\a=1}^m\la_\a.
$$
With this notation, Z. Q. Lu \cite{Lu} proved that if
$$
|B|^2+\la_2\leq n
$$
holds pointwise on $M$, then $M$ is a totally geodesic subsphere, a generalized Clifford torus, or, when $n=2$, the Veronese surface in $S^4$.


In \cite{L-Y}, we further weakened Z. Q. Lu's pinching condition by replacing the trace $|B|^2=\sum\limits_{\a=1}^m\la_\a$ with the partial sum $\sum\limits_{\a=1}^{\min\{n,m\}}\la_\a$. More precisely, we proved that if the pinching condition
$$
\sum_{\a=1}^{\min\{n,m\}}\la_\a+\la_2\leq n
$$
holds pointwise on $M$, then either $B\equiv0$ or
$$
\sum_{\a=1}^{\min\{n,m\}}\la_\a+\la_2\equiv n
$$
on $M$. This result holds in arbitrary dimension and codimension. Equality is attained by both generalized Clifford tori and Veronese manifolds. Moreover, when $n=3$, we proved that these are the only compact minimal submanifolds satisfying the equality, while the case $n=2$ follows from Z. Q. Lu's theorem \cite{Lu}.


In the present paper, we complete the classification of the equality case of the pinching condition in arbitrary dimension. The first part of the paper shows that equality in this condition at every point forces the second fundamental form to be parallel. This reduces the problem to classifying minimal submanifolds with parallel second fundamental form satisfying the same equality. More precisely, we prove the following theorem.

\begin{thm}\label{SCDK}
Let $M^n\subset S^{n+m}$ be a compact connected minimal submanifold with $n\geq2$. If
$$
\sum_{\a=1}^{\min\{n,m\}}\la_\a+\la_2\equiv n,
$$
then the second fundamental form of $M$ is parallel, i.e., $\n B\equiv0$.
\end{thm}

The key step in proving Theorem \ref{SCDK} is to show that the largest eigenvalue $\la_1$ is constant. Since $\la_1$ need not be smooth where eigenvalues coincide, we follow Z. Q. Lu \cite{Lu} and consider the smooth functions
$$
g_p:=\big(\tr(S^p)\big)^{\f{1}{p}},
$$
which converge uniformly to $\la_1$ as $p\to\infty$. Combining the matrix inequality established in \cite{L-Y} with a Simons' identity, we obtain
$$
\De g_p\geq-\ep_p,\qquad \ep_p\to0,
$$
where $\ep_p$ is independent of the point in $M$. Passing to the limit in the sense of distributions shows that $\la_1$ is weakly subharmonic. Compactness and connectedness then imply that $\la_1$ is constant. If $\la_1>\la_2$ at some point, the Laplacian formula for $\la_1$ and the Ricci equations force
$$
(\la_1,\cdots,\la_m)\equiv(n,0,\cdots,0),
$$
and Simons' identity yields $\n B\equiv0$. Otherwise, $\la_1=\la_2$ everywhere, and then
$$
\f{1}{2}\De|B|^2\geq|\n B|^2.
$$
Integrating over $M$ again proves that $\n B\equiv0$.

The second part of the paper classifies the minimal submanifolds with parallel second fundamental form for which equality holds everywhere in the pinching condition. After reducing to full immersions (i.e., replacing the ambient sphere by the smallest totally geodesic subsphere containing $M$), the classification theorem of D. Ferus \cite{F} reduces the problem to determining which standard embeddings of symmetric $R$-spaces satisfy the same equality. It therefore remains to determine the eigenvalues of their fundamental matrices.

Symmetric $R$-spaces admit several equivalent descriptions. One algebraic description regards them as compact Riemannian symmetric spaces which are also $R$-spaces, that is, as spaces which can be realized as quotients of noncompact connected semisimple Lie groups with trivial center by parabolic subgroups. This point of view is closely related to gradings of semisimple Lie algebras of the form
$$
\mf u=\mf u_{-1}\oplus\mf u_0\oplus\mf u_1,
$$
and to the classification of Kobayashi--Nagano. See \cite{K-N,K-N2,T2,T3,Q} for more details. For the computations in the present paper, we follow the conventions of T. Nagura \cite{N}. A symmetric $R$-space is considered as a compact orbit
\[
  M=K\cdot v\subset S(\mf p)
\]
of the isotropy representation associated with an orthogonal symmetric Lie algebra $(\mf g,\sigma)$ of compact type, with stabilizer $L=K_v$ such that $K/L$ is a symmetric space.


The irreducible decomposition of $(\mf g,\sigma)$ induces a corresponding product decomposition of the standard embedding. The factors are standard embeddings of irreducible symmetric $R$-spaces, scaled so that their product is minimal in the unit sphere; see Remark~\ref{rem:standard-product}. Thus the calculation can be carried out by first treating the irreducible factors and then determining how their fundamental matrices combine under this product construction.

Using the restricted root decomposition, we derive a block decomposition for the fundamental matrix of a symmetric $R$-space. We compute the eigenvalues for the irreducible cases, including both the classical and exceptional cases, and the results are collected in Tables \ref{tab:classical} and \ref{tab:exceptional}. We then determine the fundamental matrix of the scaled product in terms of those of its factors. These computations show that, among the standard embeddings of irreducible symmetric $R$-spaces, equality in the pinching condition holds everywhere precisely for the Veronese manifolds. In the product case, the equality forces exactly two totally geodesic factors, giving a generalized Clifford torus. Consequently, the minimal parallel submanifolds satisfying the equality are precisely the generalized Clifford tori and the Veronese manifolds, up to ambient isometries and totally geodesic inclusions.

Combining this classification with Theorem \ref{SCDK} and the pinching theorem in \cite{L-Y}, we obtain the following rigidity theorem.

\begin{thm}\label{main-thm}
Let $M^n\subset S^{n+m}$ be a compact minimal submanifold with $n\geq2$, and let $\la_1\geq\cdots\geq\la_m\geq0$ be the eigenvalues of its fundamental matrix. If
$$
\sum_{\a=1}^{\min\{n,m\}}\la_\a+\la_2\leq n
$$
holds pointwise on $M$, then, up to ambient isometries and totally geodesic inclusions, $M$ is one of the following:
\begin{itemize}
\item a totally geodesic subsphere, with $|B|^2\equiv0$;
\item a generalized Clifford torus, with
$$
\la_1\equiv n,\qquad \la_\a\equiv0\quad\text{for }\a\geq2;
$$
\item a Veronese manifold, with
$$
\la_1\equiv\cdots\equiv\la_{\f{(n-1)(n+2)}{2}}
\equiv\f{n}{n+1},
$$
and $\la_\a\equiv0$ for $\a>\f{(n-1)(n+2)}{2}$.
\end{itemize}
\end{thm}

	\bigskip\bigskip

\Section{Preliminaries}{Preliminaries}
\subsection{Notation and Simons' identity}

We follow the notation and curvature conventions of \cite[Section 3.1]{L-Y}. For completeness, we recall the geometric quantities and identities needed below, referring to \cite{L-Y,Lu} for the derivation of the Simons' identity. Standard background can be found in \cite{X}.

Let $f:M^n\to\ol M^{n+m}$ be an isometric immersion between Riemannian manifolds. Denote by $\ol\n$ the Levi-Civita connection of $\ol M$. The induced tangent and normal connections, both denoted by $\n$, and the \emph{second fundamental form} $B$ are given by
$$
\n_XY=(\ol{\n}_XY)^T,\qquad
\n_X\nu=(\ol{\n}_X\nu)^N,\qquad
B_{XY}=(\ol{\n}_XY)^N,
$$
where $X,Y$ are tangent vector fields, $\nu$ is a normal vector field, and the superscripts $T$ and $N$ denote the tangent and normal projections, respectively. The \emph{shape operator} in the direction $\nu$ is
$$
A^\nu(X)=-(\ol{\n}_X\nu)^T,
\qquad
\lan A^\nu(X),Y\ran=\lan B_{XY},\nu\ran.
$$

Let $\{E_1,\ldots,E_n\}$ and $\{\nu_1,\ldots,\nu_m\}$ be local orthonormal tangent and normal frames. We use the index ranges
$$
1\le i,j,k,l\le n,\qquad
1\le\a,\be,\g,\de\le m,
$$
and the summation convention unless otherwise stated. Set
$$
h_{ij}^\a:=\lan B_{E_iE_j},\nu_\a\ran,
\qquad
A^\a:=A^{\nu_\a}=(h_{ij}^\a).
$$
The \emph{mean curvature vector} is defined by
$$
H=\sum_{i=1}^n B_{E_iE_i}.
$$
The immersion is called \emph{minimal} if $H=0$, or equivalently, if $\tr A^\a=0$ for every $\a$.

We also use $\n$ for the induced connections on tensor bundles. In particular,
$$
(\n_XB)_{YZ}
=\n_X(B_{YZ})-B_{\n_XY,Z}-B_{Y,\n_XZ}.
$$
The immersion is called \emph{totally geodesic} if $B=0$, and \emph{parallel} if $\n B=0$. Following \cite{T}, with the mean curvature convention specified above, we call the immersion \emph{totally umbilic} if its mean curvature vector $H$ is everywhere nonzero and the following conditions hold:
\begin{enumerate}
\item[(i)]
For all tangent vector fields $X,Y$,
$$
B_{XY}=\f{1}{n}\lan X,Y\ran H,
$$
or equivalently, for every normal vector field $\nu$,
$$
A^\nu=\f{1}{n}\lan H,\nu\ran\Id.
$$
\item[(ii)]
The mean curvature vector is parallel with respect to the normal connection:
$$
\n_XH=0
$$
for every tangent vector field $X$.
\end{enumerate}
This convention is more restrictive than the usual definition of total umbilicity, which requires only condition~(i). In particular, totally geodesic immersions are excluded, and every totally umbilic immersion in this sense is parallel.

Our curvature convention is
$$
R_{XY}
=-\n_X\n_Y+\n_Y\n_X+\n_{[X,Y]},
$$
with the same convention for the ambient curvature $\ol{R}$. The Gauss, Codazzi and Ricci equations are
\begin{equation}\label{GCR}
\begin{aligned}
&\lan R_{XY}Z,W\ran
=\lan\ol{R}_{XY}Z,W\ran
+\lan B_{XZ},B_{YW}\ran
-\lan B_{XW},B_{YZ}\ran,\\
&(\n_XB)_{YZ}-(\n_YB)_{XZ}
=-(\ol{R}_{XY}Z)^N,\\
&\lan R_{XY}\mu,\nu\ran
=\lan\ol{R}_{XY}\mu,\nu\ran
+\lan A^\mu(X),A^\nu(Y)\ran
-\lan A^\mu(Y),A^\nu(X)\ran.
\end{aligned}
\end{equation}

For a tensor field $T$, write
$$
\n^2_{XY}T:=\n_X\n_YT-\n_{\n_XY}T.
$$
The Ricci identity for $B$ then reads
\begin{equation}\label{dXY}
\n^2_{XY}B-\n^2_{YX}B
=\n_X\n_YB-\n_Y\n_XB-\n_{[X,Y]}B
=-R_{XY}B.
\end{equation}

For the remainder of this subsection, we assume that $\ol M=S^{n+m}$ is the unit sphere and that $f$ is minimal. In particular, each $A^\a$ is symmetric and satisfies $\tr A^\a=0$.

Viewing $B$ as a bundle map $TM\otimes TM\to NM$, let $B^t$ denote its conjugate map with respect to the induced metrics. Then $\BBt:NM\to NM$ is self-adjoint and nonnegative, and
$$
\lan(\BBt)\mu,\nu\ran
=\lan B^t(\mu),B^t(\nu)\ran
=\lan A^\mu,A^\nu\ran,
$$
where the inner product of shape operators is the Hilbert--Schmidt inner product. Define the symmetric bilinear form $S$ on the normal bundle by
$$
S(\mu,\nu):=\lan A^\mu,A^\nu\ran.
$$
Its matrix in the normal frame $\{\nu_\a\}$,
$$
(S_{\a\be})
:=\big(S(\nu_\a,\nu_\be)\big)
=\big(\lan A^\a,A^\be\ran\big)
=\big(\tr(A^\a A^\be)\big),
$$
is called the \emph{fundamental matrix} of $M$. Its eigenvalues
$$
\la_1\ge\la_2\ge\cdots\ge\la_m\ge0
$$
are independent of the choice of orthonormal frames, and
$$
|B|^2
=\sum_{\a=1}^m\|A^\a\|^2
=\tr(\BBt)
=\sum_{\a=1}^m\la_\a.
$$
We identify $S$ with $\BBt$ via the normal bundle metric. In particular,
$$
\tr(S^p)=\tr\big((\BBt)^p\big)
$$
for every positive integer $p$.

Composing $f$ with a totally geodesic inclusion into a higher-dimensional unit sphere preserves minimality and enlarges the fundamental matrix $S$ by a zero block with respect to an adapted orthonormal normal frame. Thus the original eigenvalues remain unchanged, and only additional zero eigenvalues occur. We extend the sequence of eigenvalues by zeros whenever necessary.

Since the ambient sphere has constant sectional curvature $1$,
\begin{equation}\label{s1}
\lan\ol{R}_{XY}Z,W\ran
=\lan X,Z\ran\lan Y,W\ran
-\lan X,W\ran\lan Y,Z\ran,
\end{equation}
and
\begin{equation}\label{s2}
(\ol{R}_{XY}Z)^N
=\lan\ol{R}_{XY}\mu,\nu\ran=0.
\end{equation}
In particular, the Ricci equations become
$$
\lan R_{XY}\mu,\nu\ran
=\lan X,[A^\mu,A^\nu]Y\ran,
$$
where $[U,V]:=UV-VU$.

The coefficients of $\n B$ and $\n^2B$ are denoted by
$$
h_{ijk}^\a
:=\lan(\n_{E_k}B)_{E_iE_j},\nu_\a\ran,
\qquad
h_{ijkl}^\a
:=\lan(\n^2_{E_lE_k}B)_{E_iE_j},\nu_\a\ran.
$$
For each fixed $\a$, we use the notation
$$
|\n A^\a|^2
:=\sum_{i,j,k}(h_{ijk}^\a)^2,
\qquad
\De h_{ij}^\a:=\sum_k h_{ijkk}^\a,
\qquad
\De A^\a:=(\De h_{ij}^\a)_{i,j}.
$$
These expressions are formed from the full covariant derivatives of $B$, including the normal connection. In particular,
$$
|\n B|^2=\sum_{\a=1}^m|\n A^\a|^2.
$$
The induced connection on the endomorphism bundle of $NM$ satisfies
$$
\big(\n_X(\BBt)\big)\mu
=\n_X\big((\BBt)\mu\big)-(\BBt)(\n_X\mu).
$$
Since the normal connection is compatible with the metric, the identification of $S$ with $\BBt$ is preserved by covariant differentiation. Thus
\begin{equation}\label{eq:S-BBt-first}
(\n_XS)(\mu,\nu)
=\lan\big(\n_X(\BBt)\big)\mu,\nu\ran,
\end{equation}
and
\begin{equation}\label{eq:S-BBt-second}
(\n^2_{XY}S)(\mu,\nu)
=\lan\big(\n^2_{XY}(\BBt)\big)\mu,\nu\ran.
\end{equation}
Since $S=\BBt$ is symmetric, $\n_XS$ remains symmetric for every tangent vector field $X$. In particular, $\n_X(\BBt)$ is self-adjoint.

We use the notation
$$
\begin{aligned}
&\n_XS_{\a\be}
:=(\n_XS)(\nu_\a,\nu_\be),\\
&\n^2S_{\a\be}
:=\sum_k(\n^2_{E_kE_k}S)(\nu_\a,\nu_\be),
\end{aligned}
$$
and write
$$
|\n S_{\a\be}|^2
:=\sum_k\big((\n_{E_k}S)(\nu_\a,\nu_\be)\big)^2.
$$

The Simons' identity for the coefficients of $B$ takes the form
\begin{equation}\label{LaB}
\begin{aligned}
\De h_{ij}^\a
={}&nh_{ij}^\a
-h_{ik}^\be h_{kl}^\be h_{lj}^\a
+2h_{ik}^\be h_{kl}^\a h_{lj}^\be\\
&-h_{il}^\a h_{lk}^\be h_{kj}^\be
-h_{kl}^\be h_{kl}^\a h_{ij}^\be.
\end{aligned}
\end{equation}
Equivalently, in matrix notation,
\begin{equation}\label{LaB1}
\De A^\a
=nA^\a
-\sum_{\be=1}^m\lan A^\a,A^\be\ran A^\be
-\sum_{\be=1}^m[A^\be,[A^\be,A^\a]].
\end{equation}
Consequently, the Simons' identity is
\begin{equation}\label{LaB2}
\begin{aligned}
\f{1}{2}\De|B|^2
={}&h_{ijk}^\a h_{ijk}^\a
+h_{ij}^\a h_{ijkk}^\a\\
={}&|\n B|^2+n|B|^2
-\sum_{\a,\be}\lan A^\a,A^\be\ran^2
-\sum_{\a,\be}\|[A^\a,A^\be]\|^2,
\end{aligned}
\end{equation}
where $\De$ denotes the Laplace--Beltrami operator.

\subsection{A commutator inequality for symmetric matrices}
The following refined commutator inequality was established in our previous work \cite[Theorem~2.5]{L-Y}. It is valid in arbitrary dimension.

For real matrices, we use the inner product $\langle U,V\rangle=\tr(U V^T)$ and the associated norm $\|U\|^2=\langle U,U\rangle$.

\begin{lem}[{\cite[Theorem~2.5]{L-Y}}]
\label{lem:refined-commutator}
Let $n,m\geq 2$, and let $A,A_2,\ldots,A_m$ be real symmetric $n\times n$ matrices satisfying
\[
 \operatorname{tr}A=0,\qquad
 \langle A_\alpha,A_\beta\rangle=0
 \quad (2\leq\alpha<\beta\leq m),
\]
and
\[
 \|A_2\|\geq\cdots\geq\|A_m\|.
\]
Then
\begin{equation}
\label{eq:refined-commutator}
 \sum_{\alpha=2}^{m}\|[A,A_\alpha]\|^2
 \leq
 \|A\|^2\left(
   \|A_2\|^2+
   \sum_{\alpha=2}^{\min\{m,n\}}\|A_\alpha\|^2
 \right).
\end{equation}
\end{lem}

\subsection{Examples: generalized Clifford tori and Veronese manifolds}
\label{Clifford-Veronese}

We recall the generalized Clifford tori and Veronese manifolds appearing in Theorem \ref{main-thm}. The computations of their second fundamental forms and the eigenvalues of their fundamental matrices are given in \cite[Section 3.3]{L-Y}. Throughout this subsection, $n\geq 2$, and we put $\la_\a=0$ for $\a>m$.

Let $S^q(r)\subset \R^{q+1}$ denote the $q$-dimensional Euclidean sphere centered at the origin and of radius $r$. For positive integers $p,q$ with $1\leq p\leq q$ and $p+q=n$, the generalized Clifford torus
\begin{equation}
M_{p,q}:=
S^p\left(\sqrt{\f{p}{p+q}}\right)
\times
S^q\left(\sqrt{\f{q}{p+q}}\right)
\subset S^{p+q+1}
\end{equation}
is a compact minimal submanifold. The eigenvalues of its fundamental matrix satisfy
\begin{equation}
\la_1\equiv n,\qquad
\la_\a\equiv 0\quad \forall \a\geq 2.
\end{equation}

Let $\text{Sym}_{n+1}^0$ be the vector space of all traceless symmetric $(n+1)\times(n+1)$-matrices, equipped with the inner product $\lan U,V\ran=\tr(UV)$. This space has dimension $\f{n(n+3)}{2}$. The mapping $\psi:S^n\Big(\sqrt{\f{2(n+1)}{n}}\Big) \ra \text{Sym}_{n+1}^0$ defined by
\begin{equation}
\psi(x)=
\f{1}{2}\sqrt{\f{n}{n+1}}
\left(xx^T-\f{2}{n}I_{n+1}\right)
\end{equation}
is an isometric immersion whose image lies in the unit hypersphere of $\text{Sym}_{n+1}^0$. Here $x$ is seen as a column vector. Its image is called a {\it Veronese $n$-manifold}. It is a compact minimal submanifold, and the eigenvalues of its fundamental matrix satisfy
\begin{equation}
\la_1\equiv\la_2\equiv\cdots
\equiv\la_{\f{(n-1)(n+2)}{2}}
\equiv\f{n}{n+1},
\end{equation}
with $\la_\a\equiv 0$ for each $\a>\f{(n-1)(n+2)}{2}$.

Consequently, both generalized Clifford tori and Veronese manifolds satisfy
\begin{equation}
\sum_{\a=1}^{\min\{n,m\}}\la_\a+\la_2\equiv n.
\end{equation}
For the Veronese manifolds, we also have
\begin{equation}
|B|^2+\la_2
\equiv
\left(\f{(n-1)(n+2)}{2}+1\right)\f{n}{n+1}
=\f{n^2}{2}.
\end{equation}
Thus, for $n\geq 3$, the Veronese manifolds do not satisfy the pinching condition in \cite{Lu}, but attain equality everywhere in the pinching condition of Theorem~\ref{main-thm}. Theorem \ref{main-thm} shows that generalized Clifford tori and Veronese manifolds are the only compact minimal submanifolds in Euclidean spheres satisfying $\sum\limits_{\a=1}^{\min\{n,m\}}\la_\a+\la_2\equiv n$, up to isometries of the ambient spheres and totally geodesic embeddings.

\subsection{Product and composite immersions}

In this subsection, we recall the basic properties of product and composite immersions, together with the notions of fullness, following \cite{T}.

Let
$$
f_i:(M_i,g_i)\longrightarrow(\td M_i,\td g_i),\qquad 1\leq i\leq s,
$$
be isometric immersions between Riemannian manifolds, with $\dim M_i=n_i$. Consider the product manifolds
$$
(M,g)=(M_1,g_1)\times\cdots\times(M_s,g_s),
$$
and
$$
(\td M,\td g)=(\td M_1,\td g_1)\times\cdots\times(\td M_s,\td g_s),
$$
equipped with their product metrics. Let $\pi_i:M\to M_i$ denote the projection onto the $i$-th factor. The map
$$
f=f_1\times\cdots\times f_s:(M,g)\longrightarrow(\td M,\td g),
$$
defined by
$$
f(p)=\big(f_1(\pi_1(p)),\ldots,f_s(\pi_s(p))\big),\qquad p\in M,
$$
is an isometric immersion, called the \emph{product immersion}.

Denote by $B,H$ and $B_i,H_i$ the second fundamental forms and mean curvature vectors of $f$ and $f_i$, respectively. We identify tangent and normal vector fields associated with each factor with their natural lifts to the product. Then
\begin{equation}
\label{eq:BP}
B_{X_1+\cdots+X_s,\,Y_1+\cdots+Y_s}=\sum_{i=1}^s(B_i)_{X_iY_i},
\end{equation}
where $X_i,Y_i$ are tangent vector fields on $M_i$. In particular, the mixed terms of $B$ vanish. Taking the trace with respect to the product metric gives
$$
H=\sum_{i=1}^s H_i\circ\pi_i,
$$
where each summand is regarded as a normal vector field in the corresponding ambient factor.

It follows that $f$ is parallel if and only if each $f_i$ is parallel, and $f$ is minimal if and only if each $f_i$ is minimal.

We next recall how the second fundamental form and mean curvature vector change under composition with a totally geodesic or totally umbilic isometric immersion. The following result will also allow us to compare parallelism in different ambient spaces.

\begin{lem}[{\cite[Lemma~1.1]{T}}]
\label{lem:composition}
Let $(M,g)$, $(M',g')$ and $(\td M,\td g)$ be Riemannian manifolds, with $\dim M=n$ and $\dim M'=n'$. Let
$$
f':(M,g)\longrightarrow(M',g')
$$
be an isometric immersion, and let
$$
f'':(M',g')\longrightarrow(\td M,\td g)
$$
be a totally geodesic or totally umbilic isometric immersion. Consider the composite immersion
$$
f=f''\circ f':
(M,g)\xrightarrow{\,f'\,}(M',g')
\xrightarrow{\,f''\,}(\td M,\td g).
$$
Denote by $B,B',B''$ and $H,H',H''$ the second fundamental forms and mean curvature vectors of $f,f',f''$, respectively. Normal vector fields along $f'$ are identified with their images under $f''_*$. Then the following statements hold.
\begin{enumerate}
\item[(i)]
For all tangent vector fields $X,Y$ on $M$,
$$
B_{XY}
=
B'_{XY}
+\f{1}{n'}\lan X,Y\ran(H''\circ f'),
$$
and
$$
H=H'+\f{n}{n'}(H''\circ f').
$$


\item[(ii)]
The immersion $f$ is parallel if and only if $f'$ is parallel.
\end{enumerate}
\end{lem}

We next recall the notions of fullness for immersions into space forms.

An isometric immersion into a Riemannian space form is called \emph{full} if its image is not contained in any totally geodesic hypersurface of the ambient space. It is called \emph{strongly full} if it is full and its image is not contained in any totally umbilic hypersurface of the ambient space.

An isometric immersion into a Euclidean space is called \emph{substantial} if its image is not contained in any affine hyperplane. Since the totally geodesic hypersurfaces of a Euclidean space are precisely its affine hyperplanes, fullness and substantiality are equivalent in the Euclidean case.

We now apply Lemma~\ref{lem:composition} to the standard inclusion of the unit sphere into Euclidean space. Let
$$
f:M^n\longrightarrow S^N,
\qquad
\iota:S^N\hookrightarrow\mathbb R^{N+1},
\qquad
F=\iota\circ f.
$$
The inclusion $\iota$ is totally umbilic, with second fundamental form and mean curvature vector
$$
B^\iota_{UV}=-\lan U,V\ran x,
\qquad
H^\iota(x)=-Nx,
\qquad x\in S^N.
$$
Consequently,
\begin{equation}
\label{eq:B}
B^{\mathbb R}_{XY}
=
B^S_{XY}-\lan X,Y\ran F,
\qquad
H^{\mathbb R}=H^S-nF,
\end{equation}
where the superscripts indicate the ambient spaces, and $B^S,H^S$ are identified with their images under $\iota_*$. Moreover, $f$ is parallel if and only if $F$ is parallel.

\Section{Proof of Theorem 1.1}{Proof of Theorem 1.1}

Throughout this section, let $M^n$ be a compact minimal submanifold in $S^{n+m}$, with $n\geq2$, satisfying
\begin{equation}\label{eq:pinching-equality}
\sum_{\a=1}^{\min\{n,m\}}\la_\a+\la_2\equiv n.
\end{equation}

\subsection{Constancy of the largest eigenvalue}

 Let $x\in M$ be an arbitrary point, choose a local orthonormal frame field, such that the fundamental matrix at $x$ is diagonalized, i.e.
\begin{equation}
S_{\a\be}(x)=\la_\a\de_{\a\be},\quad (\BBt)\nu_{\a}=\la_\a \nu_\a,
\end{equation}
and
\begin{equation}
\la_1\geq \la_2 \cdots\geq \la_m.
\end{equation}

By (\ref{LaB1}), for each $1\leq \a\leq m$,
\begin{equation}\label{de5}\aligned
\f{1}{2}\n^2 S_{\a\a}=&\f{1}{2}(\n^2_{E_kE_k}S)(\nu_\a,\nu_\a)=h_{ijk}^\a h_{ijk}^\a+h_{ijkk}^\a h_{ij}^\a\\
=&|\n A^\a|^2+nS_{\a\a}-\sum_\be S_{\a\be}^2-\sum_\be \|[A^\a,A^\be]\|^2\\
=&|\n A^\a|^2+n\la_\a-\la_\a^2-\sum_\be \|[A^\a,A^\be]\|^2
\endaligned
\end{equation}
at $x$. Applying (\ref{eq:refined-commutator}) and (\ref{eq:pinching-equality}), we get
\begin{equation}\label{eq:eigenvalue-est}
\begin{aligned}
n\lambda_1-\lambda_1^2-\sum_\beta \|[A^1,A^\beta]\|^2 &\geq 0,\\
n\lambda_\alpha-\lambda_\alpha^2-\sum_\beta \|[A^\alpha,A^\beta]\|^2 &\geq -(\lambda_1-\lambda_2)\lambda_\alpha,\quad \a \geq 2.
\end{aligned}
\end{equation}
Substituting (\ref{de5}) and noting that $|\n A^\a|^2\geq 0$, we get
\begin{equation}\label{eq:de5-est}
\begin{aligned}
\n^2 S_{\a\a} &\geq 0,\\
\n^2 S_{\a\a} &\geq -2(\lambda_1-\lambda_2)\lambda_\alpha,\quad \a \geq 2.
\end{aligned}
\end{equation}
Since
$$\la_1 \leq \sum_{\alpha=1}^{n}\lambda_\alpha+\lambda_2=n \leq (n+1)\la_1,$$
it follows that
\begin{equation}
\f{n}{n+1}\leq \la_1\leq n.
\end{equation}

For any integer $p\geq 2$, we consider the $C^\infty$-function
\begin{equation}
f_p:=\text{tr}(S^p)
\end{equation}
and let
\begin{equation}
g_p:=(f_p)^{\f{1}{p}}.
\end{equation}
as in \cite{Lu}.We have
$$0<\la_1^p \leq f_p=\sum_{\a}\la_\a^p\leq m\la_1^p,$$
taking the $p$-th root yields
\begin{equation}
 \lambda_1\le g_p\le m^{\f{1}{p}}\lambda_1,\qquad
 \lim_{p\to \infty}\sup_M|g_p-\lambda_1|\le \lim_{p\to \infty}n\left(m^{\f{1}{p}}-1\right)=0.
 \label{eq:uniform}
\end{equation}
Applying (\ref{de5}), a straightforward calculation yields
\begin{equation}\label{f1}\aligned
\f{1}{2}\De f_p=&\f{1}{2}\sum_{\a_1,\cdots,\a_p}(\n^2 S_{\a_1\a_2}) S_{\a_2\a_3}\cdots S_{\a_p\a_1}+\cdots+\f{1}{2}\sum_{\a_1,\cdots,\a_p}S_{\a_1\a_2}S_{\a_2\a_3}\cdots (\n^2 S_{\a_p\a_1})\\
&+\sum_{\a_1,\cdots,\a_p}\sum_{j<k}S_{\a_1\a_2}\cdots\widehat{S_{\a_j\a_{j+1}}}\cdots\widehat{S_{\a_k\a_{k+1}}}\cdots S_{\a_p\a_1}\lan \n S_{\a_j\a_{j+1}},\n S_{\a_k\a_{k+1}}\ran\\
=&p\sum_\a \left(\f{1}{2}\n^2 S_{\a\a}\right)\cdot \la_\a^{p-1}+p\sum_{\a<\be}\sum_{s+t=p-2}|\n S_{\a\be}|^2\la_\a^s \la_\be^t \\
& +\f{p(p-1)}{2}\sum_\a |\n S_{\a\a}|^2\la_\a^{p-2}.
\endaligned
\end{equation}
By the Cauchy inequality, we obtain
\begin{equation}\label{f2}\aligned
|\n f_p|^2=&p^2 \sum_i\left(\sum_\a (\n_{E_i}S_{\a\a})\la_\a^{p-1}\right)^2\\
=&p^2 \sum_i\left(\sum_\a (\n_{E_i}S_{\a\a})\la_\a^{\f{p}{2}-1}\cdot \la_\a^{\f{p}{2}}\right)^2\\
\leq & p^2 f_p \sum_{\a}|\n S_{\a\a}|^2 \la_\a^{p-2}.
\endaligned
\end{equation}

Then (\ref{f1}) and (\ref{f2}) implies
\begin{equation}\aligned
\De g_p=&\f{1}{p}(f_p)^{\f{1}{p}-1}\De f_p+\f{1}{p}\left(\f{1}{p}-1\right)(f_p)^{\f{1}{p}-2}|\n f_p|^2\\
\geq &(f_p)^{\f{1}{p}-1}\sum_\a \left(\n^2 S_{\a\a}\right)\cdot \la_\a^{p-1}
\endaligned
\end{equation}
whenever $f_p(x)\neq 0$. Substituting (\ref{eq:de5-est}), we get
\begin{equation}
\label{Deg-est}
\De g_p \geq -2g_p(\lambda_1-\lambda_2)\frac{\sum\limits_{\alpha=2}^m\lambda_\alpha^p}{f_p}.
\end{equation}
Let $t=\f{\la_2}{\la_1}$, since $\la_\a\le \la_2$ for $\a \ge 2$, and $f_p\ge \la_1^p$, we have
$$\frac{\sum\limits_{\alpha=2}^m\lambda_\alpha^p}{f_p}\le (m-1)t^p.$$
In addition, $g_p\le m^{\f{1}{p}}\la_1 \le \sqrt{m}n$ for $p\geq 2$, and $\la_1-\la_2=\la_1(1-t)\le n(1-t)$. Since $(1-t)t^p$ attains its maximum on $[0,1]$ at $t=\f{p}{p+1}$, we have
$$(1-t)t^p\le\frac{p^p}{(p+1)^{p+1}}\le\frac1{p+1}.$$
Substituting these back into (\ref{Deg-est}), we obtain
\begin{equation}
\label{eq:uniformlower}
\Delta g_p\ge-\varepsilon_p,\qquad
\varepsilon_p:=\frac{2\sqrt m(m-1)n^2}{p+1}\longrightarrow0 \quad as\quad p \to \infty.
\end{equation}
By integration by parts together with the uniform convergence (\ref{eq:uniform}), letting $p\to\infty$ gives
\[
 \int_M\la_1\Delta\varphi\,dV\ge0
\]
for every nonnegative smooth function $\varphi$, i.e., $\Delta\la_1\ge0$ in the sense of distributions. Taking $\varphi\equiv1$ and using $\Delta1=0$ shows that $\int_M\la_1\Delta1\,dV=0$. Since a nonnegative distribution with vanishing integral must vanish identically, we conclude that $\Delta\la_1=0$ in the distributional sense. By elliptic regularity, $\la_1$ is smooth, so we may integrate by parts with $\la_1$ itself, obtaining
\[
 0=\int_M\la_1\Delta\la_1\,dV=-\int_M|\nabla\la_1|^2\,dV.
\]
Hence $\nabla\la_1\equiv0$. This yields the following conclusion.

\begin{lem}
\label{lambda1}
The largest eigenvalue \(\la_1\) is a positive constant on $M$.
\end{lem}

\subsection{Parallelism of the second fundamental form}

Assume $\la_1>\la_2$ holds at $x\in M$, then there exist an open set $U$ containing $x$ such that $\la_1>\la_2$ on $U$. Let $\nu_1$ be a smooth unit normal vector field on $U$ satisfying
\begin{equation}
\label{BBt1}
(B\circ B^t)\nu_1=\la_1\nu_1.
\end{equation}
Choose $\nu_2,\ldots,\nu_m$ so that $\{\nu_1,\ldots,\nu_m\}$ is an orthonormal normal frame on $U$, with
$$(B\circ B^t)\nu_\a=\la_\a\nu_\a$$
at $x$. Taking the covariant derivative of (\ref{BBt1}) along $X$, we obtain
\begin{equation}
\label{DBBt1}
\big(\n_X(B\circ B^t)\big)\nu_1+(B\circ B^t)(\n_X\nu_1)
 =(X\la_1)\nu_1+\la_1\n_X\nu_1.
\end{equation}
Taking the inner product of both sides of (\ref{DBBt1}) with $\nu_\beta\,(\beta\ge 2)$, we obtain
\begin{equation}
\n_{X}S_{1\beta}=(\la_1-\la_\beta)\lan \n_X \nu_1,\nu_{\beta}\ran.
\end{equation}
On the other hand, taking the inner product of both sides of (\ref{DBBt1}) with $\nu_1$ yields
\begin{equation}
X\la_1=\left\lan \big(\n_X (B\circ B^t)\big)\nu_1, \nu_1\right\ran.
\end{equation}
Choose a frame $\{E_1,\cdots,E_n\}$ satisfying $\n_{E_k}E_k=0$ at $x$. It follows that
\begin{equation}
\frac{1}{2}\Delta\lambda_1
 =\frac{1}{2}\nabla^2 S_{11}
 +\sum_{\beta>1}(\lambda_1-\lambda_\beta)\sum_k\langle\nabla_{E_k}\nu_1,\nu_\beta\rangle^2.
\end{equation}
Substituting (\ref{de5}) and by Lemma \ref{lambda1}, $\la_1\equiv c>0$, we have
\begin{equation}
0=|\n A^1|^2+n\lambda_1-\lambda_1^2-\sum_\beta \|[A^1,A^\beta]\|^2+\sum_{\beta>1}(\lambda_1-\lambda_\beta)\sum_k\langle\nabla_{E_k}\nu_1,\nu_\beta\rangle^2.
\end{equation}
It then follows from (\ref{eq:eigenvalue-est}) that
\begin{equation}
\label{eigenvalue1}
n\lambda_1-\lambda_1^2-\sum_\beta \|[A^1,A^\beta]\|^2=0,
\end{equation}
\begin{equation}
\langle\nabla_{E_k}\nu_1,\nu_\beta\rangle=0,
\end{equation}
for all $1\le k\le n$ and $1\le \beta \le m$, which gives $R_{XY}\nu_1=0$ for any tangent vector $X, Y$. By the Ricci equations, we further obtain
\begin{equation}
\lan R_{XY}\nu_1, \nu_\beta\ran =\lan X, [A^1, A^\beta]Y\ran=0,
\end{equation}
and hence $[A^1, A^\beta]=0$. Substituting this back into (\ref{eigenvalue1}), we obtain $\la_1(n-\la_1)=0$, and then $\la_1\equiv n$ by Lemma \ref{lambda1}. Substituting (\ref{eq:pinching-equality}) again, we get $\lambda_\alpha=0$ for $\alpha\ge2$. By the Simons' identity, we obtain
$$0=\frac{1}{2}\Delta|B|^2=|\nabla B|^2+n^2-n^2,$$
which implies
$$|\nabla B|^2=0.$$
This immediately gives the following conclusion.

\begin{lem}
\label{lem:simple-largest-eigenvalue}
If $\la_1 >\la_2$ at some point $x\in M$, then
$$(\la_1,\la_2,\cdots,\la_m)\equiv(n,0,\cdots,0),
 \qquad \n B\equiv0.$$
\end{lem}

As an immediate consequence of the preceding lemma, we obtain the following.

\begin{cor}\label{cor:eigenvalue-cases}
Exactly one of the following two cases occurs:
\begin{itemize}
\item $(\lambda_1,\ldots,\lambda_m)\equiv(n,0,\ldots,0)$ and $\nabla B\equiv0$ on $M$;
\item $\lambda_1\equiv\lambda_2$ on $M$.
\end{itemize}
\end{cor}

\begin{proof}
If $\lambda_1>\lambda_2$ at some point of $M$, then Lemma~\ref{lem:simple-largest-eigenvalue} yields the first case. Otherwise $\lambda_1\le\lambda_2$ everywhere, and since $\lambda_1\ge\lambda_2$ by definition, we obtain $\lambda_1\equiv\lambda_2$, i.e., the second case. The two cases are mutually exclusive, since $n>0$.
\end{proof}

By Corollary~\ref{cor:eigenvalue-cases}, it remains to establish the parallelism of $B$ under the assumption that $\lambda_1\equiv\lambda_2$ on $M$.

\begin{lem}
If $\la_1=\la_2$ holds everywhere on $M$, then $\n B \equiv 0$.
\end{lem}
\begin{proof}
Combining (\ref{LaB2}) and (\ref{eq:eigenvalue-est}), we obtain
\begin{equation}
\label{eq:halfDB}
\f{1}{2}\De |B|^2\ge |\n B|^2-(\la_1-\la_2)\sum_{\a \ge 2}\la_\a=|\n B|^2
\end{equation}
Since $M$ is compact, integrating (\ref{eq:halfDB}) over $M$ and using $\int_M\Delta|B|^2\,dV=0$, we get
$$0=\int_M\f{1}{2}\Delta|B|^2\,dV\ge\int_M|\n B|^2\,dV\ge0,$$
which forces $|\n B|^2\equiv0$ on $M$.
\end{proof}

Combining Corollary~\ref{cor:eigenvalue-cases} with the preceding lemma, we conclude that the second fundamental form is parallel in both cases. We therefore obtain the following theorem.

\begin{thm}
\label{SCDK1}
Let $M$ be an n-dimensional compact minimal submanifold in $S^{n+m}$ satisfying $\sum\limits_{\a =1}^{\min\{n,m\}} \la_{\a}+\la_2 \equiv n$, then $\n B\equiv 0$ on $M$.
\end{thm}

\Section{Fundamental matrices of symmetric $R$-spaces}{Fundamental matrices of symmetric $R$-spaces}
\label{sec:fundamental matrices}

\subsection{Symmetric $R$-spaces and their standard embeddings }

This subsection reviews the definition of symmetric R-spaces (see M. Takeuchi \cite{T}) together with their standard isometric embeddings.

A \emph{positive definite symmetric graded Lie algebra} consists of a real semisimple Lie algebra
$$
  \mf u
  =\mf u_{-1}\oplus\mf u_0\oplus\mf u_1,
  \qquad
  [\mf u_j,\mf u_k]\subset\mf u_{j+k},
$$
together with a Cartan involution $\tau$ such that $\tau(\mf u_j)=\mf u_{-j}$. Here $\mf u_j=0$ for $|j|>1$, we require $\mf u_{-1}\ne0$, and the adjoint action of $\mf u_0$ on $\mf u_{-1}$ is effective, in the sense that
$$
  X\in\mf u_0,\qquad
  [X,Y]=0\ \text{for all }Y\in\mf u_{-1}
  \quad\Longrightarrow\quad X=0.
$$

Let $\mc{B}$ be the Killing form of $\mf u$, and write
$$
  \mf u=\mf k\oplus\mf q,
  \qquad
  \mf k=\ker(\tau-\Id),\quad
  \mf q=\ker(\tau+\Id).
$$
The grading is generated by a unique element $e\in\mf u_0$, characterized by
$$
  [e,X]=jX\qquad (X\in\mf u_j).
$$
Since $\tau$ reverses the grading, $\tau e=-e$, so $e\in\mf q$.

Let $U$ be the connected adjoint group of $\mf u$, and let $K\subset U$ be the connected compact subgroup with Lie algebra $\mf k$. Define the \emph{stabilizer} of $e$ in $K$ by
$$
  L=\{k\in K:\operatorname{Ad}(k)e=e\},
  \qquad M=K/L.
$$
$(K,L)$ is a compact symmetric pair with decomposition
$$
  \mf k=\mf l\oplus\mf m,\qquad
  \mf l=\operatorname{Lie}(L)=\mf k\cap\mf u_0,
$$
where $\mf m$ is the $\mc{B}$-orthogonal complement of $\mf l$ in $\mf k$. Identifying $T_oM$ with $\mf m$ at $o=L$, let $g_0$ be the $K$-invariant metric determined by
$$
  (g_0)_o(X,Y)=-\mc{B}(X,Y),
  \qquad X,Y\in\mf m.
$$

The space $M=K/L$, equipped with a metric homothetic to $g_0$, is the \emph{symmetric $R$-space} associated with the above graded Lie algebra and involution. It is called \emph{irreducible} when $\mf u$ is simple.

To specify the spherical normalization, put $n=\dim M$ and equip $\mf q$ with the Euclidean inner product $\mc{B}|_{\mf q}$. The grading and the relation $\tau\mf u_{-1}=\mf u_1$ give
$$
  \dim\mf u_{-1}=\dim\mf u_1=\dim\mf m=n,
  \qquad
  \mc{B}(e,e)
  =\tr\bigl((\operatorname{ad}e)^2\bigr)=2n.
$$
Consequently, the $K$-equivariant map
\refstepcounter{equation}\label{eq:takeuchi-unit-embedding}
$$
  f:M\ra S(\mf q),\qquad
  f(kL)=\f{\operatorname{Ad}(k)e}{\sqrt{2n}},
  \eqno(\theequation)
$$
takes values in the unit sphere. Its induced metric is
\refstepcounter{equation}\label{eq:takeuchi-induced-metric}
$$
  g=f^*\mc{B}=\f{1}{2n}g_0.
  \eqno(\theequation)
$$
Indeed, for $X,Y\in\mf m$, one has $(\operatorname{ad}e)^2Y=Y$, and invariance of the Killing form yields
$$
  (df)_oX=\f{[X,e]}{\sqrt{2n}},
  \qquad
  \mc{B}([X,e],[Y,e])
  =-\mc{B}(X,[e,[e,Y]])=-\mc{B}(X,Y).
$$

With the metric \eqref{eq:takeuchi-induced-metric}, the map \eqref{eq:takeuchi-unit-embedding} is a minimal isometric embedding with parallel second fundamental form. We call $f$ the \emph{standard embedding} of $M$ in the unit sphere. In particular, if
$$
  m=\dim\mf q-n-1,
$$
then $f:(M,g)\ra S^{n+m}$ is a minimal parallel embedding.


\subsection{Restricted root decomposition}
\label{subsec:restricted-roots}

Let $M=K/L$ be a symmetric $R$-space as in the preceding subsection. We first describe the standard embedding of $M$ using the compact dual of $(\mf u,\tau)$.

Let $\mf g$ be the compact dual of $\mf u$ with respect to the Cartan involution $\tau$, and write
$$
  \mf g=\mf k\oplus\mf p,
  \qquad
  \mf p=i\mf q.
$$
The complex linear extension of $\tau$ preserves $\mf g$ and induces an involution $\si$ given by
$$
  \si(X+iY)=X-iY,
  \qquad X\in\mf k,\quad Y\in\mf q.
$$
Thus $(\mf g,\si)$ is an orthogonal symmetric Lie algebra of compact type, with eigenspace decomposition
$$
  \mf g=\mf k\oplus\mf p.
$$
The action of $K$ extends to $\mf g$, and we use the same notation $K$ for the resulting connected subgroup of $\operatorname{Aut}(\mf g)$.

We extend $\mc{B}$ complex bilinearly to the complexification of $\mf u$, retaining the same notation. Its restriction to $\mf g$ is the Killing form of $\mf g$. For the computations, fix $c>0$ and equip $\mf g$ with the $\operatorname{Aut}(\mf g)$-invariant inner product
$$
  \lan X,Y\ran=-c\,\mc{B}(X,Y),
  \qquad X,Y\in\mf g.
$$
For $X,Y\in\mf q$, we have
$$
  \lan iX,iY\ran=c\,\mc{B}(X,Y).
$$
Consequently, the map
$$
  \iota:\mf q\ra\mf p,
  \qquad
  \iota(X)=\f{iX}{\sqrt c},
$$
is a $K$-equivariant linear isometry from $(\mf q,\mc{B}|_{\mf q})$ onto $(\mf p,\lan\cdot,\cdot\ran)$.

Put
$$
  v=\iota\left(\f{e}{\sqrt{2n}}\right)
   =\f{ie}{\sqrt{2nc}}.
$$
Since $\mc{B}(e,e)=2n$, we have $\lan v,v\ran=1$. Moreover, the stabilizer of $v$ in $K$ is precisely $L$. Hence the map
$$
  \td f:M=K/L\ra S(\mf p),
  \qquad
  \td f(kL)=\operatorname{Ad}(k)v,
$$
satisfies $\td f=\iota\circ f$, where $f$ is the standard embedding constructed in the preceding subsection. In particular,
$$
  \td f^*\lan\cdot,\cdot\ran
  =f^*\mc{B}
  =\f{1}{2n}g_0.
$$
Thus $\td f$ and $f$ are congruent under the ambient linear isometry $\iota$. Their fundamental matrices agree when computed with corresponding orthonormal normal frames. Via the isometric embedding $\td f$, we regard $M=K/L$ as the orbit $K\cdot v\subset S(\mf p)$, with base point $v$.

\begin{rem}
\label{rem:standard-product}
Let $M=K\cdot v\subset S(\mf p)$ be a symmetric $R$-space with its standard embedding. Consider the decomposition
$$
(\mf g,\si)=\bigoplus_{i=1}^s(\mf g^{(i)},\si_i)
$$
into irreducible orthogonal symmetric Lie algebras, and write
$$
\mf g^{(i)}=\mf k^{(i)}\oplus\mf p^{(i)}
$$
for the corresponding eigenspace decompositions. Then
$$
\mf k=\bigoplus_{i=1}^s\mf k^{(i)},\qquad \mf p=\bigoplus_{i=1}^s\mf p^{(i)}.
$$
The associated irreducible symmetric $R$-spaces have standard embeddings
$$
M_i=K_i\cdot v_i\subset S(\mf p^{(i)}),\qquad \lan v_i,v_i\ran=1.
$$
Under the corresponding identification of the ambient Euclidean spaces,
$$
v=(r_1v_1,\ldots,r_sv_s),\qquad M=r_1M_1\times\cdots\times r_sM_s,
$$
where
$$
n_i=\dim M_i,\qquad n=\sum_{i=1}^s n_i,\qquad r_i=\sqrt{\f{n_i}{n}}.
$$
Thus the standard embedding is represented by
$$
M_1\times\cdots\times M_s\longrightarrow S(\mf p),\qquad (x_1,\ldots,x_s)\longmapsto(r_1x_1,\ldots,r_sx_s).
$$
If $g_i$ denotes the metric induced on $M_i$ by its standard embedding, then the induced metric on $M$ is
$$
g=r_1^2g_1\oplus\cdots\oplus r_s^2g_s.
$$
Writing $m_i=\dim\mf p^{(i)}-n_i-1$ for the spherical codimension of $M_i$, the spherical codimension of $M$ is
$$
m=\sum_{i=1}^s m_i+s-1.
$$
\end{rem}

In view of Remark~\ref{rem:standard-product}, we first compute the fundamental matrices of irreducible symmetric $R$-spaces. The general case will be treated in Section \ref{sec:parallel-classification} by determining how the fundamental matrix behaves under the above scaled product construction.

Now we assume that $M$ is irreducible. Choose a maximal abelian subspace $\mf a$ of $\mf p$ containing $v$, and identify $\mf a$ with its dual by means of $\lan\cdot,\cdot\ran$. For $\la\in\mf a$, define
$$
\begin{aligned}
  \mf k_\la
  &=
  \{X\in\mf k:
    \operatorname{ad}(Z)^2X
    =-\lan\la,Z\ran^2X
    \ \text{for all }Z\in\mf a\},\\
  \mf p_\la
  &=
  \{X\in\mf p:
    \operatorname{ad}(Z)^2X
    =-\lan\la,Z\ran^2X
    \ \text{for all }Z\in\mf a\}.
\end{aligned}
$$
The restricted root system is
$$
  \mf r
  =\{\la\in\mf a\setminus\{0\}:\mf p_\la\ne\{0\}\}.
$$
We choose a positive system $\mf r^+$ . Then the orthogonal restricted root decompositions are
$$
  \mf k
  =\mf k_0\oplus
   \bigoplus_{\la\in\mf r^+}\mf k_\la,
  \qquad
  \mf p
  =\mf a\oplus
   \bigoplus_{\la\in\mf r^+}\mf p_\la,
$$
where $\mf k_0$ is the centralizer of $\mf a$ in $\mf k$. We write
$$
  m_\la=\dim\mf p_\la
$$
for the multiplicity of $\la$.

We shall use the following bracket relations for the restricted root spaces.

\begin{lem}[{\cite[Chapter~VII]{H}}]
\label{lem:bracket-relation}
For $\la,\mu\in\mf r^+\cup\{0\}$, one has
$$
\begin{aligned}
  [\mf k_\la,\mf p_\mu]
  &\subset
  \mf p_{\la+\mu}+\mf p_{|\la-\mu|},\\
  [\mf k_\la,\mf k_\mu]
  &\subset
  \mf k_{\la+\mu}+\mf k_{|\la-\mu|},\\
  [\mf p_\la,\mf p_\mu]
  &\subset
  \mf k_{\la+\mu}+\mf k_{|\la-\mu|}.
\end{aligned}
$$
Here $\mf p_0=\mf a$, and $\mf k_0$ is the centralizer of $\mf a$ in $\mf k$. The notation $|\la-\mu|$ denotes the positive root among $\pm(\la-\mu)$ when $\la-\mu$ is a root, and denotes $0$ when $\la=\mu$. A summand corresponding to $\la+\mu$ or $\la-\mu$ is understood to be zero if that vector is neither zero nor a restricted root.
\end{lem}

\subsection{Shape operators and the fundamental matrix}

To compute the shape operators and the fundamental matrix, we first choose a representative of the orbit in a closed positive Weyl chamber.

\begin{lem}[{\cite[Lemma~2]{N}}]
\label{lem:weyl-chamber}
For every $X\in\mf p$, there exists an element $k\in K$ such that $kX\in\mf a$ and
$$
  \lan\la,kX\ran\geq 0
  \qquad\text{for all }\la\in\mf r^+.
$$
\end{lem}

By Lemma~\ref{lem:weyl-chamber}, after replacing $v$ by another point of the same $K$-orbit, we may assume that
$$
  v\in\mf a,
  \qquad
  \lan\la,v\ran\geq 0
  \quad\text{for all }\la\in\mf r^+.
$$
We identify $T_vM$ with the corresponding subspace of $\mf p$. Then
$$
  T_vM=[\mf k,v].
$$
Put
$$
\begin{aligned}
  \mf r_1^+
  &=\{\la\in\mf r^+:\lan\la,v\ran=0\},\\
  \mf r_2^+
  &=\{\la\in\mf r^+:\lan\la,v\ran>0\}.
\end{aligned}
$$
The tangent and normal spaces of $M$ in the unit sphere are given by
$$
  T_vM=\bigoplus_{\la\in\mf r_2^+}\mf p_\la
$$
and
$$
  T_v^\perp M
  =\mf a_v\oplus\bigoplus_{\la\in\mf r_1^+}\mf p_\la,
  \qquad
  \mf a_v=\{X\in\mf a:\lan X,v\ran=0\}.
$$

Let $A_\xi:T_vM\ra T_vM$ denote the shape operator of $M\subset S(\mf p)$ in the normal direction $\xi\in T_v^\perp M$. The following proposition is due to Takagi-Takahashi \cite{T-T}.

\begin{pro}
For $X_\la\in\mf p_\la$, $\la\in\mf r_2^+$, the shape operators are given by
$$
\begin{aligned}
  A_{Z_0}X_\la
  &=-\f{\lan\la,Z_0\ran}{\lan\la,v\ran}X_\la,
  && Z_0\in\mf a_v,\\[1.5ex]
  A_{Z_\mu}X_\la
  &=-\f{1}{\lan\la,v\ran^2}[[v,X_\la],Z_\mu],
  && Z_\mu\in\mf p_\mu,\quad\mu\in\mf r_1^+.
\end{aligned}
$$
\end{pro}

For each $\la\in\mf r^+$, there exist orthonormal bases $\{X_{\la,1},\ldots,X_{\la,m_\la}\}$ of $\mf p_\la$ and $\{Y_{\la,1},\ldots,Y_{\la,m_\la}\}$ of $\mf k_\la$ such that
\begin{equation}
\label{equ:Base}
\left\{
\begin{aligned}
  [Z,X_{\la,p}]&=-\lan\la,Z\ran Y_{\la,p},\\
  [Z,Y_{\la,p}]&=\lan\la,Z\ran X_{\la,p},
\end{aligned}
\right.
\qquad
Z\in\mf a,\quad 1\leq p\leq m_\la,
\end{equation}
where $m_\la=\dim\mf p_\la$ is the multiplicity of $\la$.

With these choices, an orthonormal basis of $T_vM$ is
$$
  \{X_{\la,p}:\la\in\mf r_2^+,\ 1\leq p\leq m_\la\}.
$$
Put $l_0=\dim\mf a_v$, and let $Z_1,\ldots,Z_{l_0}$ be an orthonormal basis of $\mf a_v$. Since $v$ is a unit vector, $v,Z_1,\ldots,Z_{l_0}$ form an orthonormal basis of $\mf a$. For each $\mu\in\mf r_1^+$, choose an orthonormal basis
$$
  Z_{\mu,1},\ldots,Z_{\mu,m_\mu}
  \quad\text{of }\mf p_\mu.
$$
An orthonormal basis of $T_v^\perp M$ is then
$$
  \{Z_1,\ldots,Z_{l_0}\}
  \cup
  \{Z_{\mu,q}:\mu\in\mf r_1^+,\ 1\leq q\leq m_\mu\}.
$$
Denote the vectors in this normal basis by $\xi_\a$. With respect to the above tangent basis, the fundamental matrix $S=(S_{\a\be})$ satisfies
\begin{equation}
\label{equ:FundamentalMatrix}
  S_{\a\be}
  =
  \sum_{\la\in\mf r_2^+}\sum_{p=1}^{m_\la}
  \lan
  A_{\xi_\a}X_{\la,p},
  A_{\xi_\be}X_{\la,p}
  \ran.
\end{equation}
By the shape operator formulas and \eqref{equ:Base}, we have
\begin{equation}
\label{equ:Weingarten}
\begin{aligned}
  A_ZX_{\la,p}
  &=-\f{\lan\la,Z\ran}{\lan\la,v\ran}X_{\la,p},
  && Z\in\mf a_v,\\
  A_{Z_{\mu,q}}X_{\la,p}
  &=\f{1}{\lan\la,v\ran}[Y_{\la,p},Z_{\mu,q}],
  && \mu\in\mf r_1^+,\quad 1\leq q\leq m_\mu,
\end{aligned}
\end{equation}
where $\la\in\mf r_2^+$ and $1\leq p\leq m_\la$. By Lemma~\ref{lem:bracket-relation},
$$
  [Y_{\la,p},Z_{\mu,q}]
  \in
  \mf p_{\la+\mu}\oplus\mf p_{|\la-\mu|}.
$$
Every nonzero summand on the right is contained in $T_vM$, since $\lan\mu,v\ran=0$ and $\lan\la,v\ran>0$.

We compute the entries of $S$ in three cases.

\emph{The $\mf a_v$-block.} For $Z_i,Z_j\in\mf a_v$, the first formula of \eqref{equ:Weingarten} gives
$$
  A_{Z_i}X_{\la,p}
  =-\f{\lan\la,Z_i\ran}{\lan\la,v\ran}X_{\la,p}.
$$
Hence
$$
\begin{aligned}
  S_{ij}
  &=
  \sum_{\la\in\mf r_2^+}\sum_{p=1}^{m_\la}
  \f{\lan\la,Z_i\ran\lan\la,Z_j\ran}
    {\lan\la,v\ran^2}\\
  &=
  \sum_{\la\in\mf r_2^+}
  m_\la
  \f{\lan\la,Z_i\ran\lan\la,Z_j\ran}
    {\lan\la,v\ran^2}.
\end{aligned}
$$

\emph{The mixed block.} Let $Z_i\in\mf a_v$ and $Z_{\mu,q}\in\mf p_\mu$, where $\mu\in\mf r_1^+$. Then
$$
  A_{Z_{\mu,q}}X_{\la,p}
  =\f{1}{\lan\la,v\ran}[Y_{\la,p},Z_{\mu,q}].
$$
It follows that
$$
  S_{i,(\mu,q)}
  =
  \sum_{\la\in\mf r_2^+}\sum_{p=1}^{m_\la}
  \left\lan
  -\f{\lan\la,Z_i\ran}{\lan\la,v\ran}X_{\la,p},
  \f{1}{\lan\la,v\ran}[Y_{\la,p},Z_{\mu,q}]
  \right\ran.
$$
The bracket term has no component in $\mf p_\la$. Indeed, $\mu\ne 0$ and $\lan\mu,v\ran=0$, whereas $\lan\la,v\ran>0$, so neither $\la+\mu$ nor $\pm(\la-\mu)$ can equal $\la$. By the orthogonality of the restricted root spaces, each summand is zero. Thus
$$
  S_{i,(\mu,q)}=0.
$$

\emph{The root space blocks.} Let $Z_{\mu,q}\in\mf p_\mu$ and $Z_{\nu,s}\in\mf p_\nu$, where $\mu,\nu\in\mf r_1^+$. By the second formula of \eqref{equ:Weingarten},
$$
  A_{Z_{\mu,q}}X_{\la,p}
  =\f{1}{\lan\la,v\ran}[Y_{\la,p},Z_{\mu,q}].
$$
Taking inner products and summing over the tangent basis gives
$$
  (S_{\mf r_1})_{(\mu,q),(\nu,s)}
  =
  \sum_{\la\in\mf r_2^+}\sum_{p=1}^{m_\la}
  \f{1}{\lan\la,v\ran^2}
  \lan
  [Y_{\la,p},Z_{\mu,q}],
  [Y_{\la,p},Z_{\nu,s}]
  \ran.
$$
Moreover, Lemma~\ref{lem:bracket-relation} gives
$$
\begin{aligned}
  [Y_{\la,p},Z_{\mu,q}]
  &\in
  \mf p_{\la+\mu}\oplus\mf p_{|\la-\mu|},\\
  [Y_{\la,p},Z_{\nu,s}]
  &\in
  \mf p_{\la+\nu}\oplus\mf p_{|\la-\nu|}.
\end{aligned}
$$
Assume that $\mu\ne\nu$. Since distinct restricted root spaces are orthogonal, the inner product above can be nonzero only if one of the roots $\la+\mu$, $|\la-\mu|$ equals one of the roots $\la+\nu$, $|\la-\nu|$. Since $\mu$ and $\nu$ are positive roots, such an equality either gives $\mu=\nu$ or gives one of the relations
$$
  2\la=\mu+\nu,\qquad
  2\la+\mu=\nu,\qquad
  2\la+\nu=\mu.
$$
The first possibility is excluded by assumption. The remaining relations are impossible after taking inner products with $v$, since
$$
  \lan\la,v\ran>0,
  \qquad
  \lan\mu,v\ran=\lan\nu,v\ran=0.
$$
Hence the inner product vanishes whenever $\mu\ne\nu$, and the blocks of $S$ between distinct normal root spaces are zero.

Combining these computations gives the following proposition.

\begin{pro}[Block decomposition for the fundamental matrix]
\label{prop:block-formula}
With respect to the normal basis
$$
  \{Z_1,\ldots,Z_{l_0}\}
  \cup
  \{Z_{\mu,q}:\mu\in\mf r_1^+,\ 1\leq q\leq m_\mu\},
$$
the fundamental matrix has the block form
\begin{equation}
\label{equ:BlockForm}
  S=
  \begin{pmatrix}
    S_{\mf a} & 0\\
    0 & S_{\mf r_1}
  \end{pmatrix}.
\end{equation}
The $\mf a_v$-block is given by
\begin{equation}
\label{equ:AspaceBlock}
  (S_{\mf a})_{ij}
  =
  \sum_{\la\in\mf r_2^+}
  m_\la
  \f{\lan\la,Z_i\ran\lan\la,Z_j\ran}
    {\lan\la,v\ran^2},
  \qquad 1\leq i,j\leq l_0.
\end{equation}
The mixed block vanishes:
\begin{equation}
\label{equ:MixedBlock}
  S_{i,(\mu,q)}=0,
  \qquad
  1\leq i\leq l_0,\quad
  \mu\in\mf r_1^+,\quad
  1\leq q\leq m_\mu.
\end{equation}
Finally, the root space block is
\begin{equation}
\label{equ:RootBlock}
  (S_{\mf r_1})_{(\mu,q),(\nu,s)}
  =
  \sum_{\la\in\mf r_2^+}\sum_{p=1}^{m_\la}
  \f{1}{\lan\la,v\ran^2}
  \lan
  [Y_{\la,p},Z_{\mu,q}],
  [Y_{\la,p},Z_{\nu,s}]
  \ran,
\end{equation}
where $\mu,\nu\in\mf r_1^+$, $1\leq q\leq m_\mu$, and $1\leq s\leq m_\nu$. If $\mu\ne\nu$, then the corresponding root space block is zero.
\end{pro}

\begin{rem}
Taking $\nu=\mu$ in \eqref{equ:RootBlock}, we denote the diagonal block on $\mf p_\mu$ by
$$
  C_\mu=((C_\mu)_{qs})_{1\leq q,s\leq m_\mu},
$$
where
\begin{equation}
\label{equ:DiagonalRootBlock}
  (C_\mu)_{qs}
  =
  \sum_{\la\in\mf r_2^+}\sum_{p=1}^{m_\la}
  \f{
    \lan
    [Y_{\la,p},Z_{\mu,q}],
    [Y_{\la,p},Z_{\mu,s}]
    \ran
  }{\lan\la,v\ran^2}.
\end{equation}
Thus
\begin{equation}
\label{equ:GeneralRootBlockForm}
  S
  =
  S_{\mf a}
  \oplus
  \bigoplus_{\mu\in\mf r_1^+}C_\mu.
\end{equation}
In the case-by-case computations below, the matrices $C_\mu$ are computed from this formula.
\end{rem}

\subsection{Examples}

We now apply the block decomposition to compute the fundamental matrices of four representative irreducible symmetric $R$-spaces before presenting the complete tables. The real, complex, and quaternionic Grassmannian examples of types AI, A, and AII are treated together, since their ambient symmetric pairs have restricted root systems of type $A_{p+q-1}$, with root multiplicities $\k=1,2,4$, respectively. We then consider an exceptional example of type EI.

\subsubsection{The Grassmannian examples of types AI, A, and AII}

Let $\ell=p+q$, where $2\leq p\leq q$. We consider the three symmetric $R$-spaces
$$
\begin{aligned}
  M_1&=SO(\ell)/S(O(p)\times O(q)),
  &&\text{of type AI},\\
  M_2&=SU(\ell)/S(U(p)\times U(q)),
  &&\text{of type A},\\
  M_4&=Sp(\ell)/(Sp(p)\times Sp(q)),
  &&\text{of type AII}.
\end{aligned}
$$
The subscript $\k\in\{1,2,4\}$ denotes the restricted root multiplicity in the corresponding example. The ambient symmetric pairs are
$$
\begin{aligned}
  (\mf g_1,\mf k_1)
  &=(\mf{su}(\ell),\mf{so}(\ell)),\\
  (\mf g_2,\mf k_2)
  &=(\mf{su}(\ell)\oplus\mf{su}(\ell),
     \De\mf{su}(\ell)),\\
  (\mf g_4,\mf k_4)
  &=(\mf{su}(2\ell),\mf{sp}(\ell)),
\end{aligned}
$$
where
$$
  \De\mf{su}(\ell)
  =\{(X,X):X\in\mf{su}(\ell)\}.
$$
Let
$$
  \mf g_\k=\mf k_\k\oplus\mf p_\k
$$
be the corresponding symmetric decomposition. We use the invariant inner products
$$
\begin{aligned}
  \lan X,Y\ran_1
  &=-\tr(XY),\\
  \lan(X_1,X_2),(Y_1,Y_2)\ran_2
  &=-\f12\tr(X_1Y_1)-\f12\tr(X_2Y_2),\\
  \lan X,Y\ran_4
  &=-\f12\tr(XY).
\end{aligned}
$$

Let
$$
  \mf a_0
  =
  \left\{
    D=\mathrm{i}\operatorname{diag}(d_1,\ldots,d_\ell):
    d_r\in\R,\quad
    \sum_{r=1}^{\ell}d_r=0
  \right\}.
$$
A maximal abelian subspace of $\mf p_\k$ is obtained as the image of $\mf a_0$ under the map
$$
  \iota_\k(D)
  =
  \begin{cases}
    D,
    &\k=1,\\[1mm]
    (D,-D),
    &\k=2,\\[1mm]
    \begin{pmatrix}
      D&0\\
      0&D
    \end{pmatrix},
    &\k=4.
  \end{cases}
$$
With the inner product $-\tr(DE)$ on $\mf a_0$, each $\iota_\k$ is an isometry. We write
$$
  \mf a_\k=\iota_\k(\mf a_0).
$$

Set
$$
  v_0
  =
  \mathrm{i}
  \begin{pmatrix}
    a\Id_p&0\\
    0&b\Id_q
  \end{pmatrix},
  \qquad
  a=\sqrt{\f{q}{p(p+q)}},
  \qquad
  b=-\sqrt{\f{p}{q(p+q)}},
$$
and let
$$
  v_\k=\iota_\k(v_0).
$$
Then $v_\k$ is a unit vector, and its orbit is $M_\k$.

For $A_D=\iota_\k(D)$, where $D=\mathrm{i}\operatorname{diag}(d_1,\ldots,d_\ell) \in\mf a_0$, define
$$
\begin{aligned}
  \lan\mu_{jk},A_D\ran_\k
  &=d_j-d_k,
  &&1\leq j<k\leq p,\\
  \lan\nu_{st},A_D\ran_\k
  &=d_{p+s}-d_{p+t},
  &&1\leq s<t\leq q,\\
  \lan\la_{js},A_D\ran_\k
  &=d_j-d_{p+s},
  &&1\leq j\leq p,\quad 1\leq s\leq q.
\end{aligned}
$$
For all three types, the positive roots split as
$$
\begin{aligned}
  \mf r_1^+
  &=
  \{\mu_{jk}:1\leq j<k\leq p\}
  \cup
  \{\nu_{st}:1\leq s<t\leq q\},\\
  \mf r_2^+
  &=
  \{\la_{js}:1\leq j\leq p,\ 1\leq s\leq q\}.
\end{aligned}
$$
Every restricted root has multiplicity $\k$, and
$$
  \lan\la_{js},v_\k\ran_\k^2
  =\f{p+q}{pq}.
$$

We next choose explicit root space bases. Let $E_{xy}$ be the standard matrix unit in $M_\ell(\C)$. For $1\leq x<y\leq\ell$, set
$$
\begin{aligned}
  R_{xy}&=E_{xy}-E_{yx},
  & I_{xy}&=\mathrm{i}(E_{xy}+E_{yx}),\\
  K_{xy}&=\mathrm{i}(E_{xy}-E_{yx}),
  & T_{xy}&=E_{xy}+E_{yx}.
\end{aligned}
$$
For a root $\rho$ determined by the pair $(x,y)$,
$$
  \lan\rho,\iota_\k(D)\ran_\k=d_x-d_y,
  \qquad
  D=\mathrm{i}\operatorname{diag}(d_1,\ldots,d_\ell).
$$
The pairs corresponding to $\mu_{jk}$, $\nu_{st}$, and $\la_{js}$ are, respectively,
$$
  (x,y)=(j,k),
  \qquad
  (x,y)=(p+s,p+t),
  \qquad
  (x,y)=(j,p+s).
$$

For type AI, where $\k=1$, choose
$$
  X_{\rho,1}^{(1)}=\f{1}{\sqrt2}I_{xy},
  \qquad
  Y_{\rho,1}^{(1)}=\f{1}{\sqrt2}R_{xy}.
$$
For type A, where $\k=2$, choose
$$
\begin{aligned}
  X_{\rho,1}^{(2)}
  &=\f{1}{\sqrt2}(R_{xy},-R_{xy}),
  &
  X_{\rho,2}^{(2)}
  &=\f{1}{\sqrt2}(I_{xy},-I_{xy}),\\
  Y_{\rho,1}^{(2)}
  &=-\f{1}{\sqrt2}(I_{xy},I_{xy}),
  &
  Y_{\rho,2}^{(2)}
  &=\f{1}{\sqrt2}(R_{xy},R_{xy}).
\end{aligned}
$$
For type AII, where $\k=4$, choose
$$
\begin{aligned}
  X_{\rho,1}^{(4)}
  &=\f{1}{\sqrt2}
    \begin{pmatrix}
      R_{xy}&0\\
      0&-R_{xy}
    \end{pmatrix},
  &
  Y_{\rho,1}^{(4)}
  &=-\f{1}{\sqrt2}
    \begin{pmatrix}
      I_{xy}&0\\
      0&-I_{xy}
    \end{pmatrix},\\[2mm]
  X_{\rho,2}^{(4)}
  &=\f{1}{\sqrt2}
    \begin{pmatrix}
      I_{xy}&0\\
      0&I_{xy}
    \end{pmatrix},
  &
  Y_{\rho,2}^{(4)}
  &=\f{1}{\sqrt2}
    \begin{pmatrix}
      R_{xy}&0\\
      0&R_{xy}
    \end{pmatrix},\\[2mm]
  X_{\rho,3}^{(4)}
  &=\f{1}{\sqrt2}
    \begin{pmatrix}
      0&R_{xy}\\
      R_{xy}&0
    \end{pmatrix},
  &
  Y_{\rho,3}^{(4)}
  &=-\f{1}{\sqrt2}
    \begin{pmatrix}
      0&I_{xy}\\
      I_{xy}&0
    \end{pmatrix},\\[2mm]
  X_{\rho,4}^{(4)}
  &=\f{1}{\sqrt2}
    \begin{pmatrix}
      0&K_{xy}\\
      -K_{xy}&0
    \end{pmatrix},
  &
  Y_{\rho,4}^{(4)}
  &=\f{1}{\sqrt2}
    \begin{pmatrix}
      0&T_{xy}\\
      -T_{xy}&0
    \end{pmatrix}.
\end{aligned}
$$
For each $\k\in\{1,2,4\}$, the vectors $X_{\rho,\a}^{(\k)}$ and $Y_{\rho,\a}^{(\k)}$, $1\leq\a\leq\k$, form orthonormal bases of $(\mf p_\k)_\rho$ and $(\mf k_\k)_\rho$, respectively. They satisfy
$$
\begin{aligned}
  [\iota_\k(D),X_{\rho,\a}^{(\k)}]
  &=
  -\lan\rho,\iota_\k(D)\ran_\k
  Y_{\rho,\a}^{(\k)},\\
  [\iota_\k(D),Y_{\rho,\a}^{(\k)}]
  &=
  \lan\rho,\iota_\k(D)\ran_\k
  X_{\rho,\a}^{(\k)}.
\end{aligned}
$$
In particular, the minus sign in $Y_{\rho,1}^{(4)}$ is required by the convention in \eqref{equ:Base}.

We first describe the orthogonal complement of $v_\k$ in $\mf a_\k$. Let
$$
  Z_z
  =
  \iota_\k\!
  \left(
    \mathrm{i}\operatorname{diag}(z_1,\ldots,z_\ell)
  \right)
  \in(\mf a_\k)_{v_\k}.
$$
Since the trace vanishes and $\lan Z_z,v_\k\ran_\k=0$, we obtain
$$
\begin{aligned}
  \sum_{j=1}^p z_j+\sum_{s=1}^q z_{p+s}&=0,\\
  a\sum_{j=1}^p z_j+b\sum_{s=1}^q z_{p+s}&=0.
\end{aligned}
$$
Since $a\ne b$, it follows that
$$
  \sum_{j=1}^p z_j=0,
  \qquad
  \sum_{s=1}^q z_{p+s}=0.
$$

For $1\leq k\leq p-1$ and $1\leq t\leq q-1$, set
$$
\begin{aligned}
  D_k^{(1)}
  &=
  \f{\mathrm{i}}{\sqrt{k(k+1)}}
  \left(
    \sum_{u=1}^k E_{uu}-kE_{k+1,k+1}
  \right),\\
  D_t^{(2)}
  &=
  \f{\mathrm{i}}{\sqrt{t(t+1)}}
  \left(
    \sum_{u=1}^t E_{p+u,p+u}
    -tE_{p+t+1,p+t+1}
  \right),
\end{aligned}
$$
and define
$$
  Z_{\k,k}^{(1)}=\iota_\k(D_k^{(1)}),
  \qquad
  Z_{\k,t}^{(2)}=\iota_\k(D_t^{(2)}).
$$
These vectors form an orthonormal basis of $(\mf a_\k)_{v_\k}$.

Let $Q_{\mf a,\k}$ be the bilinear form represented by the $(\mf a_\k)_{v_\k}$-block. For
$$
\begin{aligned}
  Z_x
  &=
  \iota_\k\!
  \left(
    \mathrm{i}\operatorname{diag}(x_1,\ldots,x_\ell)
  \right),\\
  Z_y
  &=
  \iota_\k\!
  \left(
    \mathrm{i}\operatorname{diag}(y_1,\ldots,y_\ell)
  \right)
\end{aligned}
$$
in $(\mf a_\k)_{v_\k}$, formula \eqref{equ:AspaceBlock} gives
$$
\begin{aligned}
  Q_{\mf a,\k}(Z_x,Z_y)
  &=
  \f{\k pq}{p+q}
  \sum_{j=1}^p\sum_{s=1}^q
  (x_j-x_{p+s})(y_j-y_{p+s})\\
  &=
  \f{\k pq}{p+q}
  \left(
    q\sum_{j=1}^p x_jy_j
    +p\sum_{s=1}^q x_{p+s}y_{p+s}
  \right).
\end{aligned}
$$
Consequently,
$$
\begin{aligned}
  Q_{\mf a,\k}
  (Z_{\k,j}^{(1)},Z_{\k,k}^{(1)})
  &=
  \f{\k pq^2}{p+q}\de_{jk},\\
  Q_{\mf a,\k}
  (Z_{\k,s}^{(2)},Z_{\k,t}^{(2)})
  &=
  \f{\k p^2q}{p+q}\de_{st},\\
  Q_{\mf a,\k}
  (Z_{\k,j}^{(1)},Z_{\k,t}^{(2)})
  &=0.
\end{aligned}
$$
Hence
$$
  S_{\mf a,\k}
  =
  \f{\k pq^2}{p+q}\Id_{p-1}
  \oplus
  \f{\k p^2q}{p+q}\Id_{q-1}.
$$

We next compute the root space normal blocks. Fix $\k\in\{1,2,4\}$. To simplify the notation, write
$$
  X_{\rho,\a}=X_{\rho,\a}^{(\k)},
  \qquad
  Y_{\rho,\a}=Y_{\rho,\a}^{(\k)}.
$$
For each $\mu\in\mf r_1^+$, choose the orthonormal basis of the corresponding normal root space by setting
$$
  Z_{\mu,\a}:=X_{\mu,\a}^{(\k)},
  \qquad
  1\leq\a\leq\k.
$$
Direct matrix multiplication gives
\begin{equation}
\label{equ:grassmannian-bracket-sums}
  \sum_{l=1}^p\sum_{u=1}^q\sum_{\g=1}^{\k}
  \lan
    [Y_{\la_{lu},\g},Z_{\mu,\a}],
    [Y_{\la_{lu},\g},Z_{\mu,\be}]
  \ran_\k
  =
  \begin{cases}
    \k q\,\de_{\a\be},
    &\mu=\mu_{jk},\\
    \k p\,\de_{\a\be},
    &\mu=\nu_{st}.
  \end{cases}
\end{equation}
Indeed, in the first case the bracket is nonzero only for $l=j,k$. For a fixed $\a$, there are $2\k q$ nonzero terms, each of squared norm $1/2$. In the second case it is nonzero only for $u=s,t$, giving $2\k p$ such terms. The off-diagonal sums vanish.

Using \eqref{equ:RootBlock} and $\lan\la_{lu},v_\k\ran_\k^{-2}=pq/(p+q)$, we obtain
$$
  C_{\mu_{jk}}^{(\k)}
  =
  \f{\k pq^2}{p+q}\Id_\k,
  \qquad
  C_{\nu_{st}}^{(\k)}
  =
  \f{\k p^2q}{p+q}\Id_\k.
$$

Define
$$
  d_p(\k)=(p-1)+\k\binom{p}{2},
  \qquad
  d_q(\k)=(q-1)+\k\binom{q}{2}.
$$
Combining the Cartan and root space blocks and ordering the normal basis accordingly, we obtain
\begin{equation}
\label{equ:grassmannian-series-S}
  S_\k
  =
  \f{\k pq^2}{p+q}\Id_{d_p(\k)}
  \oplus
  \f{\k p^2q}{p+q}\Id_{d_q(\k)}.
\end{equation}
Here
$$
  d_r(1)=\f{(r+2)(r-1)}{2},
  \qquad
  d_r(2)=r^2-1,
  \qquad
  d_r(4)=(2r+1)(r-1).
$$
Thus the two displayed eigenvalues have multiplicities $d_p(\k)$ and $d_q(\k)$, respectively. If $p=q$, they coincide and the multiplicities are added.


\begin{rem}[The case $p=1$]
\label{rem:type-A-p-one}
The calculation above uses $p\geq2$ only in the construction of the first-block diagonal basis $\{Z_{\k,k}^{(1)}:1\leq k\leq p-1\}$ and the roots $\{\mu_{jk}:1\leq j<k\leq p\}$. If $p=1$ and $q\geq2$, both families are empty. Formula \eqref{equ:grassmannian-series-S} remains valid with
$$
  d_1(\k)=0,
  \qquad
  d_q(\k)=(q-1)+\k\binom{q}{2}.
$$
Hence
$$
  S_\k
  =
  \f{\k q}{q+1}\Id_{d_q(\k)}.
$$
If $p=q=1$, the normal space has dimension zero, so $S_\k$ is the $0\times0$ matrix and $\operatorname{Spec}(S_\k)=\varnothing$.
\end{rem}

\subsubsection{An exceptional example of type EI}

We next consider the exceptional symmetric $R$-space of type EI. In T. Nagura's notation, this is the case where $\mf k$ is of type $C_4$ and $\mf l$ is of type $C_2\times C_2$. The restricted root system of the ambient symmetric pair is of type $E_6$, and all restricted root multiplicities are equal to $1$.

We normalize the invariant inner product so that every restricted root has squared length $2$. Let $\a_1,\ldots,\a_6$ be the simple roots of $E_6$, numbered as in the following Dynkin diagram:
$$
\setlength{\unitlength}{1em}
\begin{picture}(11,4.2)(-0.5,-1)
\put(0.2,0){\line(1,0){2.1}}
\put(2.7,0){\line(1,0){2.1}}
\put(5.2,0){\line(1,0){2.1}}
\put(7.7,0){\line(1,0){2.1}}
\put(5,0.2){\line(0,1){1.6}}
\put(0,0){\circle{0.4}}
\put(2.5,0){\circle{0.4}}
\put(5,0){\circle{0.4}}
\put(7.5,0){\circle{0.4}}
\put(10,0){\circle{0.4}}
\put(5,2){\circle{0.4}}
\put(0,-0.7){\makebox(0,0){$\a_1$}}
\put(2.5,-0.7){\makebox(0,0){$\a_2$}}
\put(5,-0.7){\makebox(0,0){$\a_3$}}
\put(7.5,-0.7){\makebox(0,0){$\a_5$}}
\put(10,-0.7){\makebox(0,0){$\a_6$}}
\put(5,2.7){\makebox(0,0){$\a_4$}}
\end{picture}
$$
Then $\a_1,\ldots,\a_5$ generate a root subsystem of type $D_5$.

Let $\om_6\in\mf a$ be determined by
$$
  \lan\a_k,\om_6\ran=\de_{k6},
  \qquad 1\leq k\leq6.
$$
With the above normalization, $\om_6$ is the sixth fundamental weight and is given by
$$
  \om_6
  =
  \f13
  \left(
    2\a_1+4\a_2+6\a_3+3\a_4+5\a_5+4\a_6
  \right).
$$
Since $\lan\om_6,\om_6\ran=4/3$, we choose
$$
  v=\f{\sqrt3}{2}\om_6.
$$
Then
$$
  \lan v,v\ran=1,
  \qquad
  \lan\a_6,v\ran=\f{\sqrt3}{2}.
$$
A positive root $\la=\sum_{k=1}^6 c_k\a_k$ belongs to $\mf r_1^+$ exactly when $c_6=0$. Choose an orthonormal basis $e_1,\ldots,e_5$ of $\mf a_v$ so that
$$
  \mf r_1^+
  =
  \De^+(D_5)
  =
  \{e_j\pm e_k:1\leq j<k\leq5\}.
$$
The roots in $\mf r_2^+$ are the sixteen roots whose projections onto $\mf a_v$ are
$$
  \f12\sum_{s=1}^5\ep_s e_s,
  \qquad
  \ep_s=\pm1,
$$
with an even number of minus signs. For all such roots,
$$
  \lan\la,v\ran=\f{\sqrt3}{2}.
$$

We first compute the $\mf a_v$-block. Taking $Z_k=e_k$, $1\leq k\leq5$, as an orthonormal basis of $\mf a_v$, formula \eqref{equ:AspaceBlock} gives
$$
\begin{aligned}
  (S_{\mf a})_{jk}
  &=
  \f43
  \sum_{\la\in\mf r_2^+}
  \lan\la,Z_j\ran\lan\la,Z_k\ran\\
  &=
  \f13\sum_{\ep}\ep_j\ep_k,
\end{aligned}
$$
where the last sum runs over the sixteen admissible sign choices. This sum equals $16$ when $j=k$, while for $j\ne k$ the positive and negative contributions cancel. Hence
$$
  S_{\mf a}=\f{16}{3}\Id_5.
$$

It remains to compute the blocks corresponding to $\mf p_\mu$, $\mu\in\mf r_1^+$. Since all restricted root multiplicities are $1$, each such block is one-dimensional. Write $C_\mu=(c_\mu)$ and omit the multiplicity index from the root space basis vectors. By \eqref{equ:DiagonalRootBlock},
$$
  c_\mu
  =
  \f43
  \sum_{\la\in\mf r_2^+}
  \|[Y_\la,X_\mu]\|^2.
$$
With the Weyl basis normalized as in Nagura's formula, if $\la+\mu$ or $\la-\mu$ is a root, the corresponding root string has length two and the relevant structure constant has square $1$. The two possibilities are mutually exclusive. Thus each contributing root $\la$ satisfies
$$
  \|[Y_\la,X_\mu]\|^2=\f12.
$$
For $\mu=e_j+e_k$, the condition for a nonzero contribution is $\ep_j=\ep_k$; for $\mu=e_j-e_k$, it is $\ep_j=-\ep_k$. In both cases, there are exactly eight admissible sign choices. Hence
$$
  c_\mu
  =
  \f43\cdot8\cdot\f12
  =
  \f{16}{3},
  \qquad \mu\in\mf r_1^+.
$$

Combining the $\mf a_v$-block and the twenty root space normal blocks, we obtain
$$
  S
  =
  \f{16}{3}\Id_5
  \oplus
  \f{16}{3}\Id_{20}
  =
  \f{16}{3}\Id_{25}.
$$

\Section{Eigenvalues of fundamental matrices of irreducible symmetric R-spaces}{Eigenvalues of fundamental matrices of irreducible symmetric R-spaces}
In this section, we list the spectra of the fundamental matrices of the non-totally geodesic standard embeddings of irreducible symmetric $R$-spaces. The notation follows T. Nagura's list of irreducible symmetric $R$-spaces.

In particular, for type $BDI(1)$, the notation $S'(O(p-1)\times O(q-1))$ denotes the subgroup of $SO(p)\times SO(q)$ consisting of pairs
$$
  \left(
    \begin{pmatrix}
      \ep&0\\
      0&A
    \end{pmatrix},
    \begin{pmatrix}
      \ep&0\\
      0&B
    \end{pmatrix}
  \right),
$$
where
$$
  A\in O(p-1),\qquad
  B\in O(q-1),\qquad
  \det A=\det B=\ep\in\{1,-1\}.
$$

The eigenvalues and their multiplicities are displayed in separate columns; in the tables, ``Mult.'' stands for multiplicity. The codimension is taken in the ambient unit sphere. It equals the size of $S$, i.e. the sum of the listed multiplicities. The spectra are obtained from the block computations in Section~\ref{sec:fundamental matrices} and determine the fundamental matrices up to orthogonal changes of the normal basis.

\subsection{Classical cases}

{\small
\setlength{\tabcolsep}{3pt}
\renewcommand{\TPTnoteSettings}{%
  \setlength{\leftmargin}{0.07\textwidth}%
  \setlength{\rightmargin}{0.07\textwidth}%
  \setlength{\labelwidth}{0.8em}%
  \setlength{\labelsep}{0.35em}%
  \setlength{\itemsep}{0.5pt}%
  \setlength{\parsep}{0pt}%
  \setlength{\topsep}{2pt}%
}

\begin{ThreePartTable}

\begin{TableNotes}
\small
\item[a] $1\leq p\leq q,\quad p+q\geq3$.
\item[b] $p\geq2$.
\item[c] $p\geq3$.
\item[d] $2\leq p\leq q,\quad p+q\geq5$.
\item[e] $p\geq3$.
\end{TableNotes}

\begin{longtable}{
  @{}
  >{\centering\arraybackslash}p{0.08\textwidth}
  >{\raggedright\arraybackslash}p{0.25\textwidth}
  >{\centering\arraybackslash}p{0.09\textwidth}
  >{\centering\arraybackslash}p{0.17\textwidth}
  >{\centering\arraybackslash}p{0.16\textwidth}
  >{\centering\arraybackslash}p{0.14\textwidth}
  @{}
}

\caption{Eigenvalues and multiplicities of the fundamental matrices for classical irreducible symmetric $R$-spaces.}
\label{tab:classical}\\

\toprule
Type &
Symmetric $R$-space $M$ &
$\dim M$ &
$\operatorname{codim}M$ &
Eigenvalue(s) &
Mult.\\
\midrule
\endfirsthead

\toprule
Type &
Symmetric $R$-space $M$ &
$\dim M$ &
$\operatorname{codim}M$ &
Eigenvalue(s) &
Mult.\\
\midrule
\endhead

\bottomrule
\endfoot

\bottomrule
\insertTableNotes\\
\endlastfoot

$A$ &
$\begin{gathered}
  SU(p+q)/{}\\
  S(U(p)\times U(q))
\end{gathered}$\tnote{a} &
$2pq$ &
$p^2+q^2-2$ &
$\begin{gathered}
  \f{2pq^2}{p+q}\\
  \f{2p^2q}{p+q}
\end{gathered}$ &
$\begin{gathered}
  p^2-1\\
  q^2-1
\end{gathered}$\\

\addlinespace

$B$ &
$\begin{gathered}
  SO(2p+1)/{}\\
  (SO(2)\times SO(2p-1))
\end{gathered}$\tnote{b} &
$2(2p-1)$ &
$(p-1)(2p-1)$ &
$4$ &
$(p-1)(2p-1)$\\

\addlinespace

$C$ &
$Sp(p)/U(p)$\tnote{b} &
$p(p+1)$ &
$(p+1)(p-1)$ &
$\displaystyle\f{p(p+2)}{2}$ &
$(p+1)(p-1)$\\

\addlinespace

$D(1)$ &
$\begin{gathered}
  SO(2p)/{}\\
  (SO(2)\times SO(2p-2))
\end{gathered}$\tnote{c} &
$4(p-1)$ &
$(p-1)(2p-3)$ &
$4$ &
$(p-1)(2p-3)$\\

\addlinespace

$D(2)$ &
$SO(2p)/U(p)$\tnote{c} &
$p(p-1)$ &
$(p-1)(p+1)$ &
$\displaystyle\f{p(p-2)}{2}$ &
$(p-1)(p+1)$\\

\addlinespace

$AI$ &
$\begin{gathered}
  SO(p+q)/{}\\
  S(O(p)\times O(q))
\end{gathered}$\tnote{a} &
$pq$ &
$\begin{gathered}
  \f{(p+2)(p-1)}{2}\\
  +\f{(q+2)(q-1)}{2}
\end{gathered}$ &
$\begin{gathered}
  \f{pq^2}{p+q}\\
  \f{p^2q}{p+q}
\end{gathered}$ &
$\begin{gathered}
  \f{(p+2)(p-1)}{2}\\
  \f{(q+2)(q-1)}{2}
\end{gathered}$\\

\addlinespace

$AII$ &
$\begin{gathered}
  Sp(p+q)/{}\\
  (Sp(p)\times Sp(q))
\end{gathered}$\tnote{a} &
$4pq$ &
$\begin{gathered}
  (2p+1)(p-1)\\
  +(2q+1)(q-1)
\end{gathered}$ &
$\begin{gathered}
  \f{4pq^2}{p+q}\\
  \f{4p^2q}{p+q}
\end{gathered}$ &
$\begin{gathered}
  (2p+1)(p-1)\\
  (2q+1)(q-1)
\end{gathered}$\\

\addlinespace

$AIII$ &
$U(p)$\tnote{b} &
$p^2$ &
$(p+1)(p-1)$ &
$\displaystyle\f{p^2}{2}$ &
$(p+1)(p-1)$\\

\addlinespace

$BDI(1)$ &
$\begin{gathered}
  (SO(p)\times SO(q))/{}\\
  S'(O(p-1)\times{}\\
  O(q-1))
\end{gathered}$\tnote{d} &
$p+q-2$ &
$(p-1)(q-1)$ &
$2$ &
$(p-1)(q-1)$\\

\addlinespace

$BDI(2)$ &
$SO(p)$\tnote{e} &
$\displaystyle\f{p(p-1)}{2}$ &
$\displaystyle\f{(p-1)(p+2)}{2}$ &
$\displaystyle\f{p(p-2)}{4}$ &
$\displaystyle\f{(p-1)(p+2)}{2}$\\

\addlinespace

$DIII$ &
$U(2p)/Sp(p)$\tnote{b} &
$p(2p-1)$ &
$(p-1)(2p+1)$ &
$p(p-1)$ &
$(p-1)(2p+1)$\\

\addlinespace

$CI$ &
$U(p)/O(p)$\tnote{b} &
$\displaystyle\f{p(p+1)}{2}$ &
$\displaystyle\f{(p-1)(p+2)}{2}$ &
$\displaystyle\f{p(p+2)}{4}$ &
$\displaystyle\f{(p-1)(p+2)}{2}$\\

\addlinespace

$CII$ &
$Sp(p)$\tnote{b} &
$p(2p+1)$ &
$(p-1)(2p+1)$ &
$p(p+1)$ &
$(p-1)(2p+1)$\\

\end{longtable}
\end{ThreePartTable}
}

\subsection{Exceptional cases}

{\small
\setlength{\tabcolsep}{3pt}

\begin{longtable}{
  @{}
  >{\centering\arraybackslash}p{0.09\textwidth}
  >{\raggedright\arraybackslash}p{0.36\textwidth}
  >{\centering\arraybackslash}p{0.10\textwidth}
  >{\centering\arraybackslash}p{0.11\textwidth}
  >{\centering\arraybackslash}p{0.14\textwidth}
  >{\centering\arraybackslash}p{0.12\textwidth}
  @{}
}

\caption{Eigenvalues and multiplicities of the fundamental matrices for exceptional irreducible symmetric $R$-spaces.}
\label{tab:exceptional}\\

\toprule
Type &
Symmetric $R$-space $M$ &
$\dim M$ &
$\operatorname{codim}M$ &
Eigenvalue(s) &
Mult.\\
\midrule
\endfirsthead

\toprule
Type &
Symmetric $R$-space $M$ &
$\dim M$ &
$\operatorname{codim}M$ &
Eigenvalue(s) &
Mult.\\
\midrule
\endhead

\bottomrule
\endfoot

$E_6$ &
symmetric space of type EIII &
$32$ &
$45$ &
$\displaystyle\f{32}{3}$ &
$45$\\

\addlinespace

$E_7$ &
symmetric space of type EVII &
$54$ &
$78$ &
$18$ &
$78$\\

\addlinespace

$EI$ &
$\begin{array}{@{}l@{}}
  \mf k\ \text{is of type }C_4,\\
  \mf l\ \text{is of type }C_2\times C_2
\end{array}$ &
$16$ &
$25$ &
$\displaystyle\f{16}{3}$ &
$25$\\

\addlinespace

$EIV$ &
$F_4/\operatorname{Spin}(9)$ &
$16$ &
$9$ &
$\displaystyle\f{16}{3}$ &
$9$\\

\addlinespace

$EV$ &
$\begin{array}{@{}l@{}}
  \mf k\ \text{is of type }A_7,\\
  \mf l\ \text{is of type }C_4
\end{array}$ &
$27$ &
$42$ &
$9$ &
$42$\\

\addlinespace

$EVII$ &
$\begin{array}{@{}l@{}}
  \mf k\ \text{is of type }\R\times E_6,\\
  \mf l\ \text{is of type }F_4
\end{array}$ &
$27$ &
$26$ &
$9$ &
$26$\\

\end{longtable}
}

\subsection{The equality case for irreducible symmetric $R$-spaces}

Section~\ref{Clifford-Veronese} shows that the Veronese manifolds satisfy \eqref{eq:pinching-equality} everywhere. We now use Tables~\ref{tab:classical} and \ref{tab:exceptional} to show that, among the standard embeddings of irreducible symmetric $R$-spaces, this equality characterizes the Veronese manifolds.

\begin{pro}
\label{pro:standard-factor-equality}
Let $M^n\subset S^{n+m}$, $n\geq2$, be the standard minimal embedding of an irreducible symmetric $R$-space. If
$$
  \sum_{\a=1}^{\min\{n,m\}}\la_\a+\la_2=n,
$$
then $M$ is congruent to a Veronese manifold.
\end{pro}

\begin{proof}
The totally geodesic case is excluded by the assumed equality. We therefore consider the cases listed in Tables~\ref{tab:classical} and \ref{tab:exceptional}. Put
$$
  r=\min\{n,m\},
  \qquad
  \Phi(M)=\sum_{\a=1}^r\la_\a+\la_2.
$$
We shall show that $\Phi(M)>n$ except for type AI with $p=1$ and type $BDI(2)$ with $p=3$.

First consider types AI, A, and AII with $p\geq2$, and write $\k=1,2,4$, respectively. Since $p\leq q$, \eqref{equ:grassmannian-series-S} gives
$$
n=\k pq,\qquad \la_1=\f{\k pq^2}{p+q}=\f{nq}{p+q}.
$$
The multiplicity of $\la_1$ is at least
$$
d_p(\k)=(p-1)+\k\binom{p}{2}\geq2.
$$
Thus $r\geq2$ and $\la_2=\la_1$, so
$$
\Phi(M)\geq\la_1+2\la_2=\f{3nq}{p+q}\geq\f32n>n.
$$

We next consider $p=1$, for which $q\geq2$. Whenever $S=c\Id_m$ with $m\geq2$, we have
$$
  \Phi(M)=(r+1)c.
$$
For type A with $p=1$, the table gives
$$
  n=2q,
  \qquad
  m=q^2-1,
  \qquad
  c=\f{2q}{q+1}.
$$
Since $m\geq n-1$ and $c>1$, it follows that
$$
  \Phi(M)=(r+1)c\geq nc>n.
$$
For type AII with $p=1$, we have
$$
  n=4q,
  \qquad
  m=(2q+1)(q-1),
  \qquad
  c=\f{4q}{q+1}.
$$
Here $m\geq n/2$ and $c\geq2$, so
$$
  \Phi(M)=(r+1)c
  \geq2(r+1)
  \geq n+2
  >n.
$$
Thus, among these three families, only type AI with $p=1$ remains.

For type AI with $p=1$, Table~\ref{tab:classical} gives
$$
  n=q,\qquad m=\f{(q-1)(q+2)}{2},\qquad S=\f{q}{q+1}\Id_m.
$$
Since $m-n=\f{(q-2)(q+1)}{2}\geq0$ for $q\geq2$, we have $r=n=q$, and hence
$$
  \Phi(M)=(r+1)\f{q}{q+1}=(q+1)\f{q}{q+1}=q=n.
$$

For the classical types
$$
  B,\quad C,\quad D(1),\quad AIII,\quad
  BDI(1),\quad DIII,\quad CI,\quad CII,
$$
Table~\ref{tab:classical} gives
$$
  S=c\Id_m,
  \qquad
  m\geq2,
  \qquad
  m\geq\f n2,
  \qquad
  c\geq2.
$$
The same inequalities hold for every exceptional case in Table~\ref{tab:exceptional}. Consequently, all these cases satisfy
$$
  \Phi(M)=(r+1)c
  \geq2(r+1)
  \geq n+2
  >n.
$$

For type $D(2)$, where $p\geq3$, we have
$$
  n=p(p-1),
  \qquad
  m=p^2-1,
  \qquad
  c=\f{p(p-2)}2.
$$
Since $m\geq n$ and $c>1$, we obtain
$$
  \Phi(M)=(n+1)c>n.
$$

It remains to consider type $BDI(2)$. For this family,
$$
  n=\f{p(p-1)}2,
  \qquad
  m=\f{(p-1)(p+2)}2=n+p-1,
  \qquad
  c=\f{p(p-2)}4.
$$
Therefore $r=n$, and
$$
  \Phi(M)-n
  =(n+1)\f{p(p-2)}4-n
  =\f{p^3(p-3)}8.
$$
Since $p\geq3$, equality holds precisely when $p=3$.

We have thus reduced the equality case to type AI with $p=1$, $q\geq2$, and type $BDI(2)$ with $p=3$. We now identify these two cases with Veronese manifolds.

For type AI with $p=1$, the isotropy representation of $SU(q+1)/SO(q+1)$ is identified with the conjugation action of $SO(q+1)$ on $\text{Sym}_{q+1}^0$. Under this identification, its standard embedding into the unit sphere is the orbit
$$
  SO(q+1)\cdot
  \operatorname{diag}
  \left(
    \sqrt{\f{q}{q+1}},
    -\f{1}{\sqrt{q(q+1)}},
    \ldots,
    -\f{1}{\sqrt{q(q+1)}}
  \right).
$$
The diagonal matrix above is $\psi(\sqrt{2(q+1)/q}\,e_1)$, where $e_1$ is the first standard basis vector of $\R^{q+1}$. Since $\psi$ is $SO(q+1)$-equivariant, this orbit is exactly the Veronese manifold defined in Section~\ref{Clifford-Veronese}.

When $q=3$, this Veronese manifold admits another realization. As observed in \cite[Remark~7.3]{O-R}, the correspondence
$$
  SU(4)/SO(4)
  \cong
  SO(6)/(SO(3)\times SO(3))
$$
identifies the corresponding isotropy representations. In the latter model, the isotropy representation is
$$
  (g,h)\cdot T=gTh^{-1},
  \qquad
  (g,h)\in SO(3)\times SO(3),
  \quad T\in\R^{3\times3}.
$$
Equip $\R^{3\times3}$ with the inner product $\lan A,B\ran=\tr(AB^T)$. The orbit through $\Id_3/\sqrt3$ is
$$
\begin{aligned}
  (SO(3)\times SO(3))\cdot\f{\Id_3}{\sqrt3}
  &=
  (SO(3)\times\{\Id_3\})\cdot\f{\Id_3}{\sqrt3}\\
  &=
  \left\{\f{1}{\sqrt3}Q:Q\in SO(3)\right\}
  \subset S^8(1).
\end{aligned}
$$
This is the standard embedding of type $BDI(2)$ with $p=3$. Under the equivalence of the isotropy representations and after normalization to the unit sphere, it is congruent to the Veronese three-manifold arising from type AI with $(p,q)=(1,3)$. Thus both remaining cases are Veronese manifolds, which completes the proof.
\end{proof}

\Section{Minimal parallel submanifolds and the proof of Theorem 1.2}{Minimal parallel submanifolds and the proof of Theorem 1.2}
\label{sec:parallel-classification}

\subsection{Ferus' classification of parallel submanifolds}
\label{subsec:ferus-decomposition}

Let $f:(M^n,g)\to S^{n+m}$ be a minimal parallel isometric immersion, with $n\geq2$. We first reduce to the case in which $f$ is full. Put
$$
V=\operatorname{span}_{\mathbb R}f(M),\qquad S(V)=V\cap S^{n+m},\qquad m_0=\dim V-n-1.
$$
Then $f$ factors as
$$
f=j\circ f_0,\qquad f_0:(M,g)\longrightarrow S(V),\qquad j:S(V)\hookrightarrow S^{n+m},
$$
where $f_0$ is full and $j$ is the totally geodesic inclusion. By Lemma~\ref{lem:composition}, $f_0$ is minimal and parallel. With respect to an adapted orthonormal normal frame,
$$
S_f=S_{f_0}\oplus0_{m-m_0}.
$$
Thus the nonzero eigenvalues are unchanged. With the convention of extending the eigenvalue sequence by zeros, the pinching expression is also unchanged.

We use the following classification theorem of D. Ferus \cite{F} (see also M. Takeuchi \cite{T}).

\begin{thm}
\label{thm:ferus-classification}
Let $f:(M^n,g)\to S^{n+m}$ be a full minimal parallel isometric immersion of a connected Riemannian manifold. Then there exist a symmetric $R$-space $(M',g')$, a standard embedding
$$
f':(M',g')\longrightarrow S^{n+m},
$$
and an isometric immersion
$$
\pi:(M,g)\longrightarrow(M',g')
$$
such that $f=f'\circ\pi$. If $(M,g)$ is complete, then $\pi$ is an isometric covering map.
\end{thm}

By Theorem~\ref{thm:ferus-classification} and Remark~\ref{rem:standard-product}, the full immersion $f_0$ factors through a standard embedding of the form
$$
F:(P,g_P)\longrightarrow S^{n+m_0},\qquad F(x_1,\ldots,x_s)=\big(r_1f_1(x_1),\ldots,r_sf_s(x_s)\big),
$$
where $f_i:(M_i,g_i)\to S^{n_i+m_i}$ are the unit-sphere standard embeddings of irreducible symmetric $R$-spaces, and
$$
P=M_1\times\cdots\times M_s,\qquad g_P=\bigoplus_{i=1}^s r_i^2g_i,\qquad r_i=\sqrt{\f{n_i}{n}},\qquad n_i=\dim M_i.
$$
Here
$$
n=\sum_{i=1}^s n_i,\qquad m_0=\sum_{i=1}^s m_i+s-1.
$$

After an orthogonal identification of the ambient spaces, the original immersion has the form
$$
f=j\circ F\circ\pi,
$$
where
$$
\pi:(M,g)\longrightarrow(P,g_P)
$$
is an isometric immersion and $j:S^{n+m_0}\hookrightarrow S^{n+m}$ is the totally geodesic inclusion. If $(M,g)$ is complete, then $\pi$ is an isometric covering map.

The decomposition may contain totally geodesic components. Since the standard embeddings are full, such a component $M_i$ is a unit sphere, with $m_i=0$ and vanishing second fundamental form. For each of the remaining components, which are not totally geodesic, the spectrum of the fundamental matrix is given in Tables~\ref{tab:classical} and \ref{tab:exceptional}.

As $\pi$ is a local isometry, we only need to determine the fundamental matrix of $F$, and then enlarge it by a zero block accounting for the additional $m-m_0$ normal directions.

\subsection{Fundamental matrices of a product of standard embeddings}

We retain the notation of the preceding subsection and consider the immersion
$$
F:(P,g_P)\ra S^{n+m_0},\qquad F(x_1,\ldots,x_s)=\bigl(r_1f_1(x_1),\ldots,r_sf_s(x_s)\bigr),
$$
where
$$
P=M_1\times\cdots\times M_s,\qquad g_P=\bigoplus_{i=1}^s r_i^2g_i,\qquad r_i^2=\f{n_i}{n}.
$$
Let $B_i$ and $S_i$ denote the second fundamental form and the fundamental matrix of $f_i$ in the unit sphere, respectively.

\begin{pro}
\label{pro:parallel-fundamental-matrix}
With respect to a suitable local orthonormal normal frame,
$$
S_F=\bigoplus_{i=1}^s\f{n}{n_i}S_i\oplus n\Id_{s-1}.
$$
Consequently, the fundamental matrix of $f=j\circ F\circ\pi$ is
$$
S_f=\bigoplus_{i=1}^s\f{n}{n_i}S_i\oplus n\Id_{s-1}\oplus0_{m-m_0}.
$$
\end{pro}

\begin{proof}
We first compute the second fundamental form of $F$. By \eqref{eq:B}, the second fundamental form of $f_i$ regarded as a Euclidean immersion satisfies
$$
(B_i^{\R})_{XY}=(B_i)_{XY}-g_i(X,Y)f_i.
$$
Consider the rescaled immersion
$$
\widehat f_i=r_if_i:(M_i,r_i^2g_i)\ra\R^{n_i+m_i+1}.
$$
Since a constant rescaling of the metric does not change the Levi-Civita connection,
$$
(B_{\widehat f_i}^{\R})_{XY}=r_i(B_i)_{XY}-r_ig_i(X,Y)f_i.
$$
Choose a local orthonormal tangent frame $e_{i1},\ldots,e_{in_i}$ on each $(M_i,g_i)$, and identify vector fields on each factor with their lifts to the product. Then
$$
E_{ik}=\f{1}{r_i}e_{ik},\qquad 1\leq i\leq s,\quad 1\leq k\leq n_i,
$$
form a local orthonormal tangent frame of $(P,g_P)$, and
$$
(B_{\widehat f_i}^{\R})_{E_{ik}E_{il}}=\f{1}{r_i}(B_i)_{e_{ik}e_{il}}-\f{\de_{kl}}{r_i}f_i.
$$

Now consider the Euclidean product immersion
$$
G=\widehat f_1\times\cdots\times\widehat f_s:(P,g_P)\ra\bigoplus_{i=1}^s\R^{n_i+m_i+1}.
$$
We regard vectors associated with each factor as vectors in the corresponding summand of the ambient Euclidean space. By \eqref{eq:BP},
$$
(B_G^{\R})_{E_{ik}E_{jl}}=0\qquad(i\ne j),
$$
and
$$
(B_G^{\R})_{E_{ik}E_{il}}=\f{1}{r_i}(B_i)_{e_{ik}e_{il}}-\f{\de_{kl}}{r_i}f_i.
$$
Since $\sum_{i=1}^s r_i^2=1$, the image of $G$ lies in the unit sphere and $G=\iota\circ F$, where $\iota:S^{n+m_0}\hookrightarrow\R^{n+m_0+1}$ is the natural inclusion. Applying \eqref{eq:B} once more gives
$$
(B_G^{\R})_{XY}=(B_F)_{XY}-g_P(X,Y)F.
$$
Hence $(B_F)_{E_{ik}E_{jl}}=0$ for $i\ne j$, while
\refstepcounter{equation}\label{eq:product-second-form}
$$
(B_F)_{E_{ik}E_{il}}=\f{1}{r_i}(B_i)_{e_{ik}e_{il}}+\de_{kl}\left(F-\f{1}{r_i}f_i\right).
\eqno(\theequation)
$$

We next choose a normal frame adapted to \eqref{eq:product-second-form}. For each $i$, let $\nu_{i1},\ldots,\nu_{im_i}$ be a local orthonormal normal frame of $f_i$ in the unit sphere. Viewed as vectors in the corresponding ambient summand, they are unit normal vectors of $F$, and those belonging to different factors are mutually orthogonal.

Moreover, a vector $(a_1f_1,\ldots,a_sf_s)$ is orthogonal to the tangent space of the product, and it is orthogonal to $F$ if and only if $\sum_{i=1}^s r_ia_i=0$. Such vectors form an $(s-1)$-dimensional space. Choose constants $a_{\rho i}$ satisfying
$$
\sum_{i=1}^s r_ia_{\rho i}=0,\qquad \sum_{i=1}^s a_{\rho i}a_{\si i}=\de_{\rho\si},\qquad 1\leq\rho,\si\leq s-1,
$$
and put
$$
\xi_\rho=(a_{\rho1}f_1,\ldots,a_{\rho s}f_s).
$$
Then $\xi_1,\ldots,\xi_{s-1}$ are orthonormal normal vector fields of $F$, orthogonal to all $\nu_{i\a}$. Since $m_0=\sum_{i=1}^s m_i+s-1$, the vector fields $\nu_{i\a}$ and $\xi_\rho$ together form a local orthonormal normal frame of $F$.

Let $A_i^\a$ be the matrix of the shape operator of $f_i$ in the direction $\nu_{i\a}$ with respect to the frame $\{e_{ik}\}$. Since $\nu_{i\a}$ is orthogonal to $F$ and to all the position vectors $f_j$, formula \eqref{eq:product-second-form} gives
$$
\lan(B_F)_{E_{ik}E_{il}},\nu_{i\a}\ran=\f{1}{r_i}\lan(B_i)_{e_{ik}e_{il}},\nu_{i\a}\ran.
$$
The same formula gives zero when both tangent vectors belong to another factor, and all mixed tangent components vanish. Thus, with respect to the frame $\{E_{ik}\}$,
$$
A^{\nu_{i\a}}=\operatorname{diag}\left(0,\ldots,\f{1}{r_i}A_i^\a,\ldots,0\right).
$$
It follows that
$$
\tr\bigl(A^{\nu_{i\a}}A^{\nu_{i\be}}\bigr)=\f{1}{r_i^2}\tr\bigl(A_i^\a A_i^\be\bigr)=\f{n}{n_i}(S_i)_{\a\be}.
$$
For $i\ne j$, these shape operators act on different tangent components, so
$$
\tr\bigl(A^{\nu_{i\a}}A^{\nu_{j\be}}\bigr)=0.
$$

For the remaining normal directions, the values of $B_i$ are orthogonal to $\xi_\rho$, and
$$
\lan F,\xi_\rho\ran=0,\qquad \lan f_i,\xi_\rho\ran=a_{\rho i}.
$$
Therefore, \eqref{eq:product-second-form} yields
$$
\lan(B_F)_{E_{ik}E_{il}},\xi_\rho\ran=-\f{a_{\rho i}}{r_i}\de_{kl},
$$
and hence
$$
A^{\xi_\rho}=-\operatorname{diag}\left(\f{a_{\rho1}}{r_1}\Id_{n_1},\ldots,\f{a_{\rho s}}{r_s}\Id_{n_s}\right).
$$
Consequently,
$$
\tr\bigl(A^{\xi_\rho}A^{\xi_\si}\bigr)=\sum_{i=1}^s\f{n_i}{r_i^2}a_{\rho i}a_{\si i}=n\sum_{i=1}^s a_{\rho i}a_{\si i}=n\de_{\rho\si}.
$$
Finally, each $f_i$ is minimal, so $\tr A_i^\a=0$, and the mixed terms satisfy
$$
\tr\bigl(A^{\nu_{i\a}}A^{\xi_\rho}\bigr)=-\f{a_{\rho i}}{r_i^2}\tr A_i^\a=0.
$$
Combining these identities proves the formula for $S_F$. Since $\pi$ is a local isometry and the totally geodesic inclusion $j$ adds $m-m_0$ normal directions with zero shape operators, the formula for $S_f$ follows.
\end{proof}

\subsection{Classification of the equality case}

We now combine Proposition~\ref{pro:parallel-fundamental-matrix} with Proposition~\ref{pro:standard-factor-equality} to classify the minimal parallel immersions satisfying \eqref{eq:pinching-equality} everywhere.

\begin{thm}
\label{thm:parallel-equality}
Let $f:(M^n,g)\to S^{n+m}$ be a minimal parallel isometric immersion, with $n\geq2$. Let $\la_1\geq\cdots\geq\la_m\geq0$ be the eigenvalues of its fundamental matrix, extended by zeros when necessary. Then
$$
\sum_{\a=1}^{\min\{n,m\}}\la_\a+\la_2=n
$$
holds everywhere if and only if, up to ambient isometries and totally geodesic inclusions, $f(M)$ is a generalized Clifford torus or a Veronese manifold.
\end{thm}

\begin{proof}

By Proposition~\ref{pro:parallel-fundamental-matrix}, the fundamental matrix of $f$ has the form
$$
S_f=\bigoplus_{i=1}^s\f{n}{n_i}S_i\oplus n\Id_{s-1}\oplus0_{m-m_0},
$$
where $S_i$ is the fundamental matrix of the $i$-th factor in its standard embedding.

Suppose first that $s=1$. Then $n_1=n$ and
$$
S_f=S_1\oplus0_{m-m_0}.
$$
Hence the irreducible factor also satisfies \eqref{eq:pinching-equality} everywhere, and Proposition \ref{pro:standard-factor-equality} shows that it is congruent to a Veronese manifold.

Now suppose that $s\geq2$. The block $n\Id_{s-1}$ shows that $\la_1\geq n$. Since all eigenvalues are nonnegative,
$$
n=\sum_{\a=1}^{\min\{n,m\}}\la_\a+\la_2\geq\la_1+\la_2\geq n+\la_2.
$$
Consequently,
$$
\la_1=n,\qquad \la_2=0,
$$
so $S_f$ has rank one. The block $n\Id_{s-1}$ already has rank $s-1$, and therefore
$$
s=2,\qquad S_1=S_2=0.
$$
Both factors are thus totally geodesic. Since their standard embeddings are full, they are unit spheres. Writing $p=n_1$ and $q=n_2$, we obtain
$$
f(M)=S^p\left(\sqrt{\f{p}{n}}\right)\times S^q\left(\sqrt{\f{q}{n}}\right)\subset S^{n+1},\qquad p+q=n,
$$
up to an ambient isometry and a totally geodesic inclusion. Hence $f(M)$ is a generalized Clifford torus.

Conversely, the spectra given in Section~\ref{Clifford-Veronese} show that generalized Clifford tori and Veronese manifolds satisfy \eqref{eq:pinching-equality} everywhere. Totally geodesic inclusions only add zero eigenvalues and therefore preserve this equality.
\end{proof}

\subsection{Proof of Theorem 1.2}

By \cite{L-Y}, either
$$\sum_{\alpha=1}^{\min\{n,m\}}\lambda_\alpha+\lambda_2
\equiv0$$
or (\ref{eq:pinching-equality}) holds everywhere on $M$. In the former case, the nonnegativity of the eigenvalues implies $B\equiv0$, so $M$ is totally geodesic. In the latter case, Theorem \ref{SCDK1} gives $\nabla B\equiv0$. Applying the classification result in Theorem \ref{thm:parallel-equality}, we conclude that $M$ is a generalized Clifford torus or a Veronese manifold. For generalized Clifford tori and Veronese manifolds, the eigenvalues of the fundamental matrix are exactly those listed in the theorem. This completes the proof.

	\bigskip\bigskip

	\bibliographystyle{amsplain}

\end{document}